\documentclass[11pt]{article}
\usepackage{calrsfs, amsmath, amscd, amsfonts, amssymb, amsthm, tikz, times, bbm, bm}
\usepackage[T1]{fontenc}
\usepackage{listings}
\usepackage{authblk}
\usepackage{xy}
\xyoption{all}

\newcommand\independent{\protect\mathpalette{\protect\independenT}{\perp}}
\def\independenT#1#2{\mathrel{\rlap{$#1#2$}\mkern2mu{#1#2}}}

\newtheorem{thm}{Theorem}[section]
\newtheorem{cor}[thm]{Corollary}
\newtheorem{lem}[thm]{Lemma}
\newtheorem{prop}[thm]{Proposition}
\theoremstyle{definition}
\newtheorem{defi}[thm]{Definition}
\newtheorem*{remark}{Remark}

\newtheorem{example}[thm]{Example}

\def\ed{\stackrel{d}{=}}
\def\bi{\begin{itemize}}
\def\ei{\end{itemize}}
\def\aa{\alpha}
\def\bb{\beta}

\definecolor{drab}{rgb}{0.59,0.44,0.09}
\def\tcd{\textcolor{drab}}
\def\be{\begin{equation}}
\def\ee{\end{equation}}
\def\bea{\begin{eqnarray}}
\def\eea{\end{eqnarray}}
\def\nn{\nonumber}

\def\({\left(}
\def\){\right)}
\def\[{\left[}
\def\]{\right]}

\newcommand*{\defeq}{\stackrel{\mathrm{def}}{=}}
\newcommand*{\sep}{\mathit{sep}}

\def\t{\tilde}

\def\d={\buildrel d \over =}
\def\rstr{\mathit{Restr}}
\def\srstr{\mathit{Restr}^o}
\def\dom{\mathit{Dom}}

\def\X{\mathbb{X}}

\def\N{\mathbb{N}}
\def\Q{\mathbb{Q}}

\def\Z{\mathbb{Z}}

\def\M{\mathcal{M}}

\def \P{\mathbf{P}}
\def \bS{\mathbf{S}}
\def \bU{\mathbf{U}}

\def \bX{\mathbf{X}}
\def \bY{\mathbf{Y}}

\def \AA{{\cal A}}

\def \CC{{\cal C}}

\def \EE{{\cal E}}
\def \FF{{\cal F}}
\def \GG{{\cal G}}
\def \HH{{\cal H}}

\def \MM{{\cal M}}
\def \NN{{\cal N}}

\def \TT{{\cal T}}

\def \XX{{\cal X}}

\newcommand{\hy}[1]{\tcd{[{\textbf{HY:}\textit{#1}}]}}

\newcommand{\indices}[1]{\mathbb{N}^{#1}}
\newcommand{\res}[2]{#1|_{#2}}
\newcommand{\commentout}[1]{}

\begin{document}
\title{A Generalization of Hierarchical Exchangeability on Trees to Directed Acyclic Graphs}

\author[1]{Paul Jung}
\author[1]{Jiho Lee}
\author[2]{Sam Staton}
\author[3]{Hongseok Yang}

\affil[1]{Department of Mathematical Sciences, KAIST}
\affil[2]{Department of Computer Science, University of Oxford}
\affil[3]{School of Computing, KAIST}

\maketitle

\abstract{  
		Motivated by the problem of designing inference-friendly Bayesian nonparametric models
		in probabilistic programming languages, we introduce a general class of partially exchangeable random arrays  which generalizes the notion of hierarchical exchangeability introduced in Austin and Panchenko (2014). We say that our partially exchangeable arrays are DAG-exchangeable since their partially exchangeable structure is governed by a collection of Directed Acyclic Graphs. More specifically, such a random array is indexed by $\N^{|V|}$ for some DAG $G=(V,E)$, and its exchangeability structure is governed by the edge set $E$.  We prove a representation theorem for such arrays which generalizes the Aldous-Hoover and  Austin-Panchenko representation theorems.

\vspace{5mm}

Key words: Bayesian nonparametrics, Exchangeability, Hierarchical exchangeability, Aldous-Hoover representation, de Finetti representation

\section{Introduction}

In \cite{austin:panchenko:2014}, Austin and Panchenko consider a random array indexed by {$\ell$-tuples of} paths over a collection of $\ell$ infinitely-branching rooted trees, of finite depths $\{r_1,\ldots, r_\ell\}$, where each path in the $\ell$-tuple starts from the root of one tree in the collection and ends at a leaf of that tree.\footnote{They studied such a random array to address an issue arising in spin glasses, which is to prove the predictions of the M{\'e}zard and Parisi ansatz for diluted spin glass models~\cite{Mzard2001TheBL}. See~\cite{panchenko2015hierarchical} for details.} 
If the branches emanating from a given vertex are labelled by $\N$, then the index set of the random array is $\N^{r_1}\times\cdots\times\N^{r_\ell}$. This random array is {\it hierarchically exchangeable}, defined in \cite{austin:panchenko:2014},  if its joint distribution remains invariant under rearrangements that preserve the structure of each {rooted} tree in the collection underlying the index set. {In other words, if its joint distribution remains invariant under any map
$\tau=(\tau_1,\ldots,\tau_\ell)$ on the index set, where $\tau_i$ is a rooted-graph isomorphism of the $i$th rooted tree.}

In their work, Austin and Panchenko prove that such arrays have a representation in the spirit of the celebrated Aldous-Hoover representation for exchangeable arrays of random variables \cite{hoover1979relations, aldous1981representations, aldous1985exchangeability}.  In the special case where all trees in the collection have a depth of one, {$r_1=\cdots=r_\ell=1$}, i.e., are copies of $\N$ rooted at $\emptyset$ (see {\bf Figure~\ref{fig:tree}}), then hierarchical exchangeability reduces to separate exchangeability, also known as row-column exchangeability. The number of trees in the collection corresponds to the dimension of the random array. We refer to \cite[Ch. 7]{kallenberg2006probabilistic} (see also \cite{aldous1985exchangeability, austin2012exchangeable}) for the definition of separate exchangeability, {a statement of the Aldous-Hoover theorem,} and additional background on what are now classic results in the theory of exchangeable random arrays.

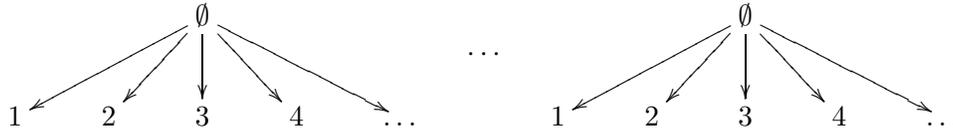
\begin{figure}
        $$
        \xymatrix{
                  &   & \emptyset \ar[dll] \ar[dl] \ar[d] \ar[dr] \ar[drr] \\
                1 & 2 & 3 & 4 & \ldots \\
        }
        \begin{array}{l}
        \\
        \\
        \quad\cdots\quad
        \end{array}
        \xymatrix{
                  &   & \emptyset \ar[dll] \ar[dl] \ar[d] \ar[dr] \ar[drr] \\
                1 & 2 & 3 & 4 & \ldots \\
        }
        $$
        \caption{A collection of rooted trees of depth 1}
        \label{fig:tree}
\end{figure} 

{In this work, we generalize hierarchical exchangeability in \cite{austin:panchenko:2014} from trees to directed acylic graphs (DAGs). Our generalization is motivated by issues related to exchangeable random processes from Bayesian nonparametrics and their implementations in probabilistic programming languages,} {which we will describe in a separate subsection of this introduction.} 
	
	Let us describe our generalization of hierarchical exchangeability. Recall that a finite DAG is the same thing as a finite partially ordered set (that is, a set with a binary relation $\preceq$ that is reflexive, transitive, and anti-symmetric).  For a given DAG, we use the partial order on its vertices: $v\preceq w$ whenever there is a path (possibly of length zero) from~$v$ to~$w$. 
	{Conversely, we can regard a finite partially ordered set as a DAG having a directed edge $(v,w)$ if $v \prec w$ whenever there is no $v'$ such that $v\prec v'\prec w$.}
	
	In the sequel, we slightly abuse notation by writing $G$ when we refer to the vertex set of $G$. Conversely, when we refer to a subset or some other vertex set, we will often assume an underlying edge structure induced by the partial order.
	A subset~$W$ of a DAG $G$ is \textbf{downward-closed}
	(with respect to the partial ordering) if $w'\in W$ and $w\preceq w'$ implies that $w\in W$.
	We will see that, in the context of hierarchical exchangeability, being ``higher'' in the hierarchy than $v$ corresponds to being smaller than $v$ in the partial order $\preceq$.  We will often say ``closed'' instead of ``downward-closed'' for simplicity.
	For a closed subset $C$ of $G$, we denote the collection of 
	closed subsets of $C$ by
	$$
	\AA_C  \defeq \{D ~:~ D \text{   closed\, and\, }D \subseteq C\}.
	$$
	\begin{defi}\label{def:randomarray}
		Let $C\in\AA_G$. A \textbf{$C$-type multi-index} is a function $\aa$ from the vertices in $C$ to $\mathbb{N}$. We write $\N^C$ for the set of all $C$-type multi-indices.  If $D\in\AA_C$, then every $C$-type multi-index $\aa\in\N^{C}$ can be restricted to a $D$-type multi-index $\aa|_{D}\in \N^{D}$. 
		A \textbf{$C$-type random array} in a Borel space $\XX$ is a family of random variables, $$\bX_C = (X_{\aa} : \aa \in \indices C)$$ indexed by $C$-type multi-indices. Here each $X_\aa$ is a random variable taking values in $\XX$.
	\end{defi}
	\begin{remark}
	\noindent
	We henceforth do not mention the Borel space $\XX$, and just say $C$-type random array.
	\end{remark}
	\begin{defi}\label{def: G-auto}
		Let $G$ be a finite DAG. A \textbf{$G$-automorphism} of $\N^G$ is a bijection $\tau\colon \N^G\to\N^G$ such that 
		\begin{equation}
		\label{eqn: automorphism} 
		\res \aa C = \res \bb C \Longleftrightarrow \res{(\tau(\aa))}C = \res {(\tau (\bb))} C,\quad
                \text{for all $C\in\AA_G$ and all $\aa,\bb \in \N^G$}.
		\end{equation}		
		For $C\in\AA_G$, we say that a $C$-type random array $\bX_C$ is \textbf{DAG-exchangeable} if for every $G$-automorphism $\tau$,
		$$ \bX_C\ed \tau(\bX_C) $$
                where $\tau(\bX_C)$ is the random array $(\bX_{\tau(\alpha)} : \alpha \in \mathbb{N}^C)$.
	\end{defi}
		 
	\begin{remark}
        \noindent The index set of a $C$-type random array, $\mathbb{N}^C$, also has a natural DAG structure which is infinitely-branching for every nonterminal vertex. Since this structure is somewhat complicated, we will delay its explanation until Example \ref{ex:infdag}.
\end{remark} 

        DAG-exchangeability is an instance of {\it partial exchangeability} which means that it does not require invariance with respect to all bijections of the index set $\mathbb{N}^G$, but rather, only with respect to some subgroup of these bijections \cite{aldous1985exchangeability}. In the case of         DAG-exchangeability, this subgroup is precisely the set of
        $G$-automorphisms. In Examples \ref{ex:known} and \ref{ex:G-auto}, we present examples of bijections which are or are-not $G$-automorphisms.  

Our main result is a representation theorem for DAG-exchangeable arrays. 
The Aldous-Hoover theorem represents an exchangeable array $\bX = (X_{ij})$ with a  measurable function of collections of independent uniform random variables corresponding to the entries in the array and the symmetries of the array, as well as an additional uniform random variable corresponding to the symmetries of the symmetries. Our representation is similar in that it represents arrays with a measurable function of collections of independent uniform random variables. Again, one collection corresponds to the entries in the array, and the other collections correspond to the symmetries of the array induced by DAG-exchangeability, as well as the symmetries of the symmetries.

\begin{thm}\label{cor 2}
	If a $G$-type random array $\bX$ is DAG-exchangeable, then there exists a measurable function $g: [0,1]^{\AA_G} \to \XX$ such that
	\begin{equation}
	\label{eqn: representation}
	\bX = \left(X_\alpha : \alpha \in \indices G\right) \ed \Big(g\big(U_{\alpha|_C} : C \in \AA_G\big) : \alpha \in \indices G\Big)
	\end{equation}
	where 
	 the $U_\beta$ are independent $[0,1]$-uniform random variables.
\end{thm}

In fact, in this work we prove a slightly more general result than the above, namely that `consistent' representations can be found  for each $C$-type array, simultaneously for all $C\in\AA_G$.
\begin{defi}\label{def:fg-randomarray}
	Let $\CC$ be a sequence consisting of distinct closed sets of $G$'s vertices. A \textbf{$\CC$-type random array collection} in $\XX$ is a sequence of random variable families $$\mathbb{X}= (\bX_C : C \in \CC)$$ where each $\bX_C$ is a $C$-type random array.
\end{defi}
In particular, for every DAG-exchangeable random array with indices in $\mathbb{N}^G$ and each downward-closed subset $C$ of $G$,
	there is a canonical way of generating a random array with the index set $\mathbb{N}^C$. 
	   Our main representation theorem provides representations for all such induced random arrays, simultaneously. 
	We will refer to such a family of consistent representations, for all induced $C$-type arrays, as being  {\it fine-grained}.

There are many reasons for considering such generalized random arrays as we do here. Partial exchangeability was considered by de Finetti himself. For instance in \cite[Ch. 12]{de1975theory}, he discusses its role in both parametric and nonparametric Bayesian statistics.  In the 1980s the subject flourished\footnote{ Two somewhat recent surveys of the early results in this field are given in \cite{austin2012exchangeable} and \cite{aldous2009more}.}, and $d$-dimensional random arrays (i.e., matrices and tensors) emerged as fundamental structures underlying the theory of partial exchangeability. Indeed, in the foundational work of \cite{hoover1979relations}, separately exchangeable arrays together with their joint and weak exchangeable counterparts, are seen to arise quite naturally as mathematical objects. However, even there,
the question of when representations arise for other partially exchangeable random arrays is posed in Section 7. As already mentioned, \cite{austin:panchenko:2014} is one work in this direction using probabilistic arguments in the spirit of \cite{aldous1981representations} and \cite{kallenberg2006probabilistic}. Also, using ultraproducts and other model theoretic tools in the spirit of \cite{hoover1979relations}, the work of \cite{crane2017relative, crane2018relatively} introduce the quite general notion of {\it relative} exchangeability, from which many forms of partial exchangeability can be extracted. 

The main contribution of this work is the extension, in terms of a probabilistic proof, of the concrete framework of hierarchical exchangeability. As noted in the works of Crane and Towsner, hierarchical exchangeability can already be extracted from their more abstract framework; we will show in the appendix that DAG-exchangeability for single arrays also falls under the umbrella of their abstract framework and thus can also be (nontrivially) realized in that framework (we have yet to see whether DAG-exchangeability for families of arrays, which is addressed in our main result, fits into the picture of relative exchangeability, albeit there are some indications that it should). On the level of applications of DAG-exchangeability, we provide a summary of our motivations with respect to Bayesian nonparametric models and probabilistic programming in the next subsection of this introduction. 

Let us mention that recently, partial exchangeability was found to have ramifications in the study of random graphs, their limits, and their statistical properties \cite{austin2008exchangeable, diaconis2008graph, veitchroy2015sparse, caicampbellbroderick2016edge, caron2017sparse, crane2018edge}. We remark that the exchangeability of graphs is not the topic of this paper-- instead we use directed graphs as a tool to create and describe our probabilistic symmetries. It is however, not unreasonable to envision that our results can be applied to this line of research in the future.


Following the applications presented in the next subsection, the rest of the paper is organized as follows.  {In Section \ref{sec:setup}, we present some examples which motivate our notion of DAG-exchangeable arrays, and illustrate the probabilistic symmetries induced by $G$-automorphisms. Our examples include, in particular, how hierarchical exchangeability fits into the framework of DAG-exchangeability. In Section \ref{sec:result} we start by extending the notion of DAG-exchangeability to collections of $C$-type random arrays, and then present our main result. The proof of this result comprises Section \ref{sec:proof}.} In the appendix we indicate an alternate route of proving our representation, without the fine-graining discussed above. This alternate method is model-theoretic and is based on the work of \cite{crane2017relative}.


\subsection{Applications to Probabilistic Programming}
\label{sec:probprog}

In terms of applications, our motivation comes from studying generative models of array-like structures through probabilistic programming languages\footnote{At this stage, Church~\cite{church2008} and Anglican~\cite{wood-aistats-2014,TolpinMYW16} have some support for advanced Bayesian nonparametric models, through the XRP feature and the \texttt{produce}/\texttt{absorb} constructs for random processes, but we are really thinking of a next generation of probabilistic programming languages, e.g.~\cite{statonetal2017pps}, with proper module and library functionality.}. These are high-level languages for statistical modeling that come equipped with separate Bayesian inference engines, which implement statistical inference algorithms such as the Metropolis-Hastings algorithm and Gibbs sampling.

In that context, one application of exchangeability and Aldous-Hoover type theorems is to identify when an elaborate, hierarchical generative model can be replaced by an equivalent one, with better independence properties, that is more amenable to inference engines. To briefly summarize this, we provide a concrete illustration in the case of a 2-dimensional exchangeable random array.  A statistical programmer can implement this as an abstract data type with the following functions:
\begin{equation}\label{eqn:datatype}
\begin{minipage}{0.92\linewidth}
\begin{lstlisting}
  newArray: () -> Array              
  newRow: Array -> Row
  newColumn: Array -> Column
  entry: (Array,Row,Column) -> real  
\end{lstlisting}
\end{minipage}
\end{equation}
One can build a finite part of the array by writing a program such as:
\begin{lstlisting}
  a = newArray();
  r1 = newRow(a); ... ; rm = newRow(a);
  c1 = newColumn(a); ... ; cn = newColumn(a);
  result[1,1] = entry(a,r1,c1); ... ; result[1,n] = entry(a,r1,cn); 
  result[2,1] = entry(a,r2,c1); ... ; result[2,n] = entry(a,r2,cn); 
  ...
  result[m,1] = entry(a,rm,c1); ... ; result[m,n] = entry(a,rm,cn)
\end{lstlisting}
so that the program will randomly return an $m\times n$ array which is a projection of the ultimate infinite random array. 

To be more precise, we suppose that the rows correspond to people, and the columns to movies, and that there is a~$1$ in a particular cell of an array if that person liked that movie, and a~$0$ otherwise. 
In this scenario, the `infinite relational model'~\cite{IRM,IHRM,OrbanzRoy} provides a reasonable array.
The model implicitly clusters the people into sorts, and clusters movies into genres,
and there is an implicit probability that each sort of person will like a movie in each genre. 
This can be described in a generative way, where we build the array up, as follows~\cite[Ex.~IV.1]{OrbanzRoy}. 
\begin{itemize}
\item An \lstinline|Array| will be represented by a mutable data structure in memory that records how many people are currently grouped into each sort, and how many movies are currently in each genre, as well as how many people of each sort liked how many movies of each genre. 
\item \lstinline|newRow| works according to the `Chinese Restaurant Process' \cite{ghosal2017fundamentals} as follows. Suppose that there are currently $m$ sorts of people, and
$c_i$ people of each sort. Then the new person will be of sort $i\leq m$ with probability
$
  \frac {c_i}{1+\sum_{j=1}^mc_j}
$
and a new sort with probability $\frac 1 {1+\sum_jc_j}$.
\item When we add a new person (row), we need to also update the new entries in the array -- does the new person like the existing movies (columns)?
We do this using a Polya urn scheme for each genre~$j$. 
Suppose that the new person is of sort~$i$.
We now run through the movies in genre~$j$, and like each one in turn with probability
$$\frac{a+1}{a+b+2}$$
where $a$ is the number of `likes' between people of sort $i$ and movies in genre~$j$, including the movies considered for the new person so far,
and $b$ is the number of `dislikes'.
\item
Similarly, \lstinline|newColumn| works according to the Chinese Restaurant Process, picking a genre for a new movie,
and deciding whether the existing people like the new movie. 
\end{itemize}
This implementation is a natural generative program, but the conditional independence appears to be complicated because each step depends on
the last. Because of the exchangeability properties of this infinite relational model,
the Aldous-Hoover theorem provides an equivalent implementation that is less generative
but with much clearer independence properties. In this particular example we can use a stick-breaking implementation of the
Chinese Restaurant Process, and a Beta distribution to implement the Polya urn~\cite[Ex.~IV.6]{OrbanzRoy}. 
\begin{itemize}
\item An \lstinline|Array| is represented by two random countable partitions of $[0,1]$, say $r,c:[0,1]\to \mathbb N$, chosen by stick breaking according to the Dirichlet process, together with a random function
  $e:\mathbb N^2\to[0,1]$, i.i.d. uniformly distributed, giving the probability of liking for any given sort of person and genre. The functions $r,c,$ and $e$ are determined by the empirical distribution of the array in question.
\item 
  Each person and movie is assigned a uniformly distributed number in $[0,1]$. 
\item
  The probability that a person $p$ likes a movie $m$ is $e(r(p),c(m))$.
  One way to set this up is to say that each person-movie pair ($p$, $m$) is assigned a uniform random $s_{m,p}\in[0,1]$ i.i.d.,~and the person $p$ likes the movie~$m$ if $s_{m,p}<e(r(p),c(m))$. 
\end{itemize}

The Aldous-Hoover theorem guarantees that the above procedure provides an accurate sample from the array in question.
Putting this together, we have a random function $F: [0,1]^3\to \{0,1\}$ given by
$F(p,m,s_{m,p})=[s_{m,p}<e(r(p),c(m))]$. 
Picking a parameterization for the random function $F$ gives an ordinary function
$f:[0,1]^4\to \{0,1\}$, which is an Aldous-Hoover representation. 

An inference engine can take advantage of the clear independence between people and movies in this implementation.

From the interface~\eqref{eqn:datatype} we can directly read off the DAG for random arrays given in Example~\ref{ex:known}(b). 
{Furthermore, the interface suggests generalizations that are more complicated and hierarchical, but which are natural to consider in a generative model. For instance, we may extend the interface in \eqref{eqn:datatype} such that each array cell contains not just a real number but also an array:}
\begin{lstlisting}
  newArray: () -> Array              
  newRow: Array -> Row
  newColumn: Array -> Column
  entry: (Array,Row,Column) -> real  

  newNestedArray: (Array,Row,Column) -> NestedArray
  newNestedRow: NestedArray -> NestedRow
  newNestedColumn: NestedArray -> NestedColumn
  nestedEntry: (NestedArray,NestedRow,NestedColumn) -> real 
\end{lstlisting}
Then, we may change the first implementation of the infinite relational model such that 
a person is represented by a nested row while a row denotes a sort of people; similarly,
a movie is represented by a nested column while a column denotes a genre of movies.
From this interface we can directly read off the DAG for random block matrices given in Example~\ref{ex:unknown}(b).

{If such a complex generative model is suitably exchangeable (see also~\cite{statonetal2017pps,statonetal2018icalp}), then our theorem says that it could just have well been implemented by sampling from uniform distributions, and thus the conditional independence relationships can be properly understood. For instance, for the nested version of the infinite relational model from above, our theorem provides its equivalent variant with better independence properties that provide the definitions of both \lstinline|entry| and \lstinline|nestedEntry| operations. Note that the \lstinline|entry| operation depends only on the row and column indices of the outer matrix, while the \lstinline|nestedEntry| operation uses all of the row and column indices of the outer and nested matrices. Only a fine-grained representation theorem is able to handle operations depending on different parts of the indices.}

\begin{remark}
  In the 2-dimensional situation there are many further things to be said about programming and inference with Aldous-Hoover representations and generative models, for instance:
  \begin{itemize}\item in terms of computability, the Aldous-Hoover representation of a generative model need not be computable in general  \cite{ackermanetal2018graphons,freer-roy2012apal};
  \item in terms of estimation, by placing some smoothness assumptions on the functions involved, there are optimal algorithms for estimating the Aldous-Hoover representation (e.g.~\cite{glz-graphon}). In particular, methods such as those in ~\cite{glz-graphon} exemplify the usefulness of having functional representations of the type pursued in this work. 
  \end{itemize}
  It would be interesting and useful to extend these developments to more complex DAG-exchangeable settings, but we will not explore this further here.
\end{remark}

\section{Examples}\label{sec:setup}
\newcommand{\hide}[1]{}

In this section we consider examples of DAG-exchangeable random arrays (Definition~\ref{def:randomarray}) and a fine-grained generalization (Definition~\ref{def:fg-randomarray}).

Our first example shows that
DAG-exchangeability generalizes several popular notions of exchangeability from the literature, when we choose $G$ appropriately. 

\begin{example}\label{ex:known}
        Most discretely-indexed stochastic processes can be viewed as $G$-type random arrays for some DAG $G$. We illustrate this perspective with basic examples from the literature.
	\begin{itemize}
                \item[(a)] {\it de Finetti Sequences:} The most common discretely-indexed stochastic processes are $G$-type random arrays where $G$ is the graph with only a single vertex $v$ and no edges. The index set $\indices G$ in this case is $\mathbb{N}^{\{v\}} \simeq \mathbb{N}$. Thus, these $G$-type random arrays are $\mathbb{N}$-indexed families of random variables. {Every permutation on $\mathbb{N}$ is a $G$-automorphism.} In this case, DAG-exchangeability becomes the standard notion of exchangeability for random sequences in de Finetti's classic result.
		
		\item[(b)] {\it Aldous-Hoover Arrays:} The graph with two vertices ${r,c}$ and no edges
                        corresponds to infinite random matrices or random arrays indexed by $\N^2$. In other words, the {multi-}index set $\indices G$ is $\mathbb{N}^{\{r,c\}} \simeq \mathbb{N}^2$. Thus, a $G$-type {multi-}index is a pair of two numbers, one denoting a row index and the other denoting a column index. These $G$-type random arrays have the form $(X_{n,m} : n,m \in \mathbb{N})$ and are random matrices with countably many rows and columns. {A $G$-automorphism $\tau$ corresponds to a pair of permutations $\pi,\pi'$ on $\mathbb{N}$, with $\pi$ acting on the row index and $\pi'$ acting on the column index.} In this case, DAG-exchangeability becomes Aldous-Hoover (separate) exchangeability.
		
		\item[(c)] {\it Hierarchical Exchangeability:} Let $r_1,\ldots,r_\ell$ be nonnegative integers. Austin and Panchenko studied a stochastic process indexed by a tuple of paths over $\ell$ countably-branching trees that have heights $r_1,\ldots,r_\ell$, respectively~\cite{austin:panchenko:2014}. Formally, this process is a family of random variables of the following form:
		$$
                        (X_\aa : \aa \in \mathbb{N}^{r_1} \times \ldots \times \mathbb{N}^{r_\ell}).
		$$
                This stochastic process is a $G$-type array for the DAG $G$ in {\bf Figure~\ref{fig:DAG:austin-panchenko}}. 
                So 
                $$
                        G = {\{v^{(i)}_j \mid 1 \leq i \leq \ell,  1 \leq j \leq r_i\}}
                $$ 
                and
		$$
		\mathbb{N}^G
		\simeq \mathbb{N}^{r_1} \times \ldots \times \mathbb{N}^{r_\ell}.
		$$
		Thus, $(X_\aa : \aa \in \mathbb{N}^{r_1} \times \ldots \times \mathbb{N}^{r_\ell})$ is the same thing as a $G$-type array. In this case, DAG-exchangeability is the same thing as hierarchical exchangeability.
	\end{itemize}
\end{example}

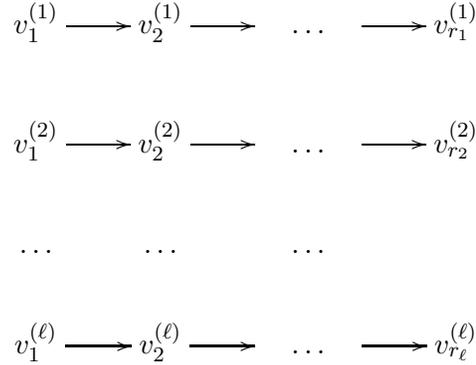
\begin{figure}
	$$
	\xymatrix{
		v^{(1)}_1 \ar[r]& v^{(1)}_2 \ar[r]&\quad\ldots\quad \ar[r]& v^{(1)}_{r_1} &\\
		v^{(2)}_1 \ar[r]& v^{(2)}_2 \ar[r]&\quad\ldots\quad \ar[r]& v^{(2)}_{r_2} &\\
		\ldots &\ldots & \ldots & \\
		v^{(\ell)}_1 \ar[r]& v^{(\ell)}_2 \ar[r]&\quad\ldots\quad \ar[r]& v^{(\ell)}_{r_\ell} &\\
	}
	$$ 
	\caption{The DAG for multi-path-indexed random arrays in \cite{austin:panchenko:2014}.}
	\label{fig:DAG:austin-panchenko}
\end{figure}

Of course, our {framework} is not limited to just recasting well-known exchangeable stochastic processes. Its recipe for defining multi-indices via a DAG makes it easy to define a random array with unusual multi-indices. Furthermore, by moving from random arrays to random array collections, we can express multiple random-variable families whose multi-index sets are related.


\begin{example}\label{ex:unknown}
	In order to illustrate the generality of our setting, we present some other instances of DAG-exchangeable arrays that are not mainstream in the exchangeability literature. We mention again that the framework of \cite{crane2017relative} and \cite{crane2018edge} can be used to derive DAG-exchangeability as in our appendix. Thus these examples, except for the last part of (b), fit into their abstract framework as well.
\begin{itemize}
\item[(a)]{\it Sequences of Random Matrices:}
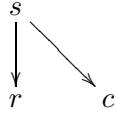
\begin{figure}[ht]
	$$ 
	\xymatrix{ 
		s\ar[d]\ar[rd] &  \\ 
		r & c}
	$$
	\caption{The DAG for sequences of random matrices.}
	\label{fig:DAG:matrix sequence}
\end{figure}
                The multi-index set for the DAG $G$ in {\bf Figure~\ref{fig:DAG:matrix sequence}} is \mbox{$\indices G = \mathbb{N}^{\{s,r,c\}}$} which can be thought of as an infinite sequence (arbitrarily labeled) of matrices of the type found in Example \ref{ex:known}(b).   For $\aa \in \indices G$, the number $\aa(s)$ determines which matrix to look at, while $\aa(r)$ and $\aa(c)$ are the row and column numbers of the matrix. When these matrices are DAG-exchangeable, the sequence is an exchangeable sequence, and each matrix is a separately exchangeable Aldous-Hoover array. Note that two entries, say $X^{(s_1)}_{r_0,c_0}$ and $X^{(s_2)}_{r_0,c_0}$, of two different arrays in this sequence, which are in the same position $(r_0,c_0)$, are only related through the exchangeability of the sequence $(s_1,s_2,...)$, and not through their position $(r_0,c_0)$. Thus this structure has a different partial exchangeability than a three-dimensional separately exchangeable Aldous-Hoover array. For instance, DAG-exchangeability permits the use of different permutations for the rows and columns of $X^{(s_1)}$ and those of $X^{(s_2)}$, while the three dimensional separate exchangeability forbids it.

Sequences of operators naturally arise in mathematical physics-- see for instance Ch. 6.2 of \cite{bratteli1996operator}. A specific example is an exchangeable  sequence of non-Hermitian random matrices with separately exchangeable entries \cite{ bordenave2011nonHermitian, bordenave2012around}.

 

\item[(b)] {\it Random Block Matrices:} 
\begin{figure}[ht]
	\centering
	\includegraphics[height=2.4in]{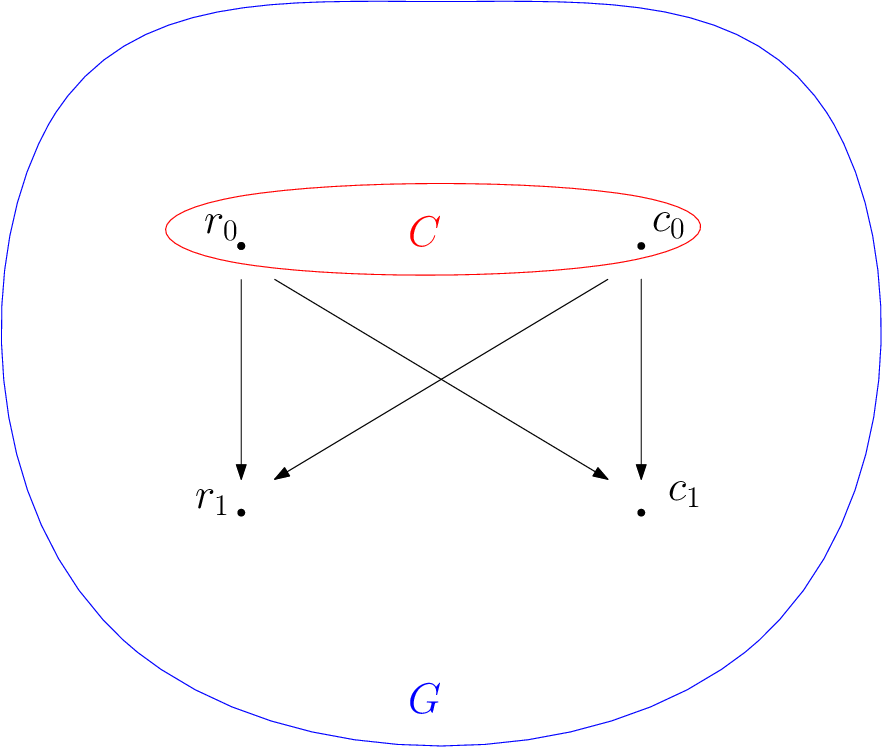}
	\caption{The DAG for random block matrices; $\CC=(C,G)$.}
	\label{fig:DAG:random-block-matrices}
\end{figure}
%
Consider the DAG $G$ in {\bf Figure~\ref{fig:DAG:random-block-matrices}}.
                The multi-index set is $\indices G = \mathbb{N}^{\{r_0,c_0,r_1,c_1\}}$, which can be understood as indices of an infinite matrix each of whose entry is again an infinite matrix. For $\aa \in \indices G$, the pair $(\aa(r_0),\aa(c_0))$ specifies the row and column of the outer matrix, and $(\aa(r_1),\aa(c_1))$ those of the nested matrix. Thus, in a random $G$-type array, each random variable $X_\aa$ stores the value of the $(\aa(r_1),\aa(c_1))$-th entry of the nested matrix, which is itself stored at the $(\aa(r_0),\aa(c_0))$-th entry of the outer nesting matrix. {We want to point out that} if in Example~\ref{ex:known}(c), one takes $\ell=2$ and $r_1=r_2=2$, then the multi-index set is the same as in this example. Thus, as a stochastic process, this example is just a special case of Example~\ref{ex:known}(c). However, as an {\it exchangeable} stochastic process, this relationship no longer holds. The presence of the additional directed edges here means that a random $G$-type array should satisfy more symmetries than Example~\ref{ex:known}(c), particularly those symmetries that are expected to hold for exchangeable random block matrices.
                

Now, let $C=\{r_0,c_0\}$, a closed subset of $G$. A small generalization of the random block matrix, similar to our example in Section~\ref{sec:probprog}, is a random structure that is simultaneously a random matrix (Example~\ref{ex:known}(b)) and a random block matrix. This can be thought of as a random matrix where each cell contains both a value in $\mathcal X$ and another random matrix. It comprises both a $C$-type random array and a $G$-type random array. In other words, it is a $(C,G)$-type random array collection.

\item[(c)]{\it Random Block Matrices and Sequences:}
One can use all the previous examples to build new examples. For instance, in {\bf Figure \ref{fig:bm-seq1}}, all three DAGs give the multi-index set $\mathbb{N}^{\{s, r_0,c_0,r_1,c_1\}}$. The edge sets of the three DAGs, however, lead to three different random structures under DAG-exchangeability. The left side can be thought of as a sequence of random block matrices. But if one removes the edge from $s\to c_0$ (see {\bf Figure \ref{fig:bm-seq2}}), then the DAG-exchangeable array is better 
 thought of as one single random block matrix with rows subject to two-level hierarchical exchangeability.  The right  figure can be thought of as a block matrix such that in each entry of each inner matrix, one finds a sequence of random variables (rather than a single random variable), thus it is a block matrix of sequences. If one removes the edge from $c_1\to s$ in the middle figure, then one can still view it as a block matrix of sequences, but the distribution of the sequence no longer depends on which `inner' column it is associated with.  
 Finally, the figure on the right has many natural interpretations in terms of DAG-exchangeability. We invite the reader to ponder upon the interesting different interpretations for the associated DAG-exchangeable arrays in this case.
 
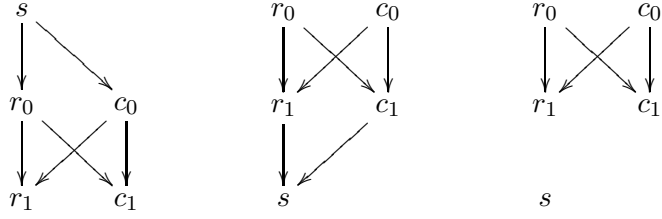
\begin{figure}[ht]
	$$ 
	\xymatrix{ 
		s\ar[d]\ar[rd] & \\ 
                r_0\ar[d]\ar[rd] & c_0\ar[d]\ar[ld] \\
                r_1 & c_1}
        \qquad\qquad
	\xymatrix{ 
                r_0\ar[d]\ar[rd] & c_0\ar[d]\ar[ld] \\
                r_1\ar[d] & c_1\ar[ld]\\ 
                s &}
        \qquad\qquad
	\xymatrix{ 
                r_0\ar[d]\ar[rd] & c_0\ar[d]\ar[ld] \\
                r_1 & c_1\\
                s &}
	$$
	\caption{The DAGs for different extensions of random block matrices.}
        \label{fig:bm-seq1}
\end{figure}
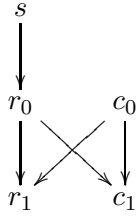
\begin{figure}[ht]
	$$ 
	\xymatrix{ 
                s\ar[d] & \\ 
                r_0\ar[d]\ar[rd] & c_0 \ar[d]\ar[ld] \\
                r_1 & c_1}
	$$
	\caption{The DAG for random block matrices subject to two-level hierarchical exchangeability.}
        \label{fig:bm-seq2}
\end{figure}

\item[(d)] {\it Random Walls:} 
Here is another example of a random array collection with {$\CC \neq (G)$}. Consider the graph $G$ consisting of three vertices $x,y,z$ and no edges. {Define $\CC$ as follows:} 
$$
                \CC = (C_{xy}, C_{yz}, C_{zx}),\qquad
                C_{xy} = \{x,y\},\qquad
                C_{yz} = \{y,z\},\qquad
                C_{zx} = \{z,x\}.
$$
A $\CC$-type random array collection consists of three random variable families, namely, $\bX_{C_{xy}}$, $\bX_{C_{yz}}$ and $\bX_{C_{zx}}$ (see {\bf Figure \ref{fig:walls}}). These families use different yet related multi-index sets, $\mathbb{N}^{\{x,y\}}$, $\mathbb{N}^{\{y,z\}}$, and $\mathbb{N}^{\{z,x\}}$, respectively. A good way to understand this array collection is to imagine a $3$-dimensional grid at points in $\mathbb{N}^{\{x,y,z\}}$. The collection associates a random variable for each point in the $xy$, $yz$ and $zx$ planes with the respective missing coordinate set to $0$ (see {\bf Figure \ref{fig:wallpic}}). Viewing the tuple $\mathbb{X}=(\bX_{C_{xy}},\, \bX_{C_{yz}}, \bX_{C_{zx}})$ in this way, rather than just as three $2$-dimensional random arrays, makes it easy to state and study symmetries which involve all three families, as we explain soon.
\begin{figure}[ht]
	\centering
	\includegraphics[height=2in]{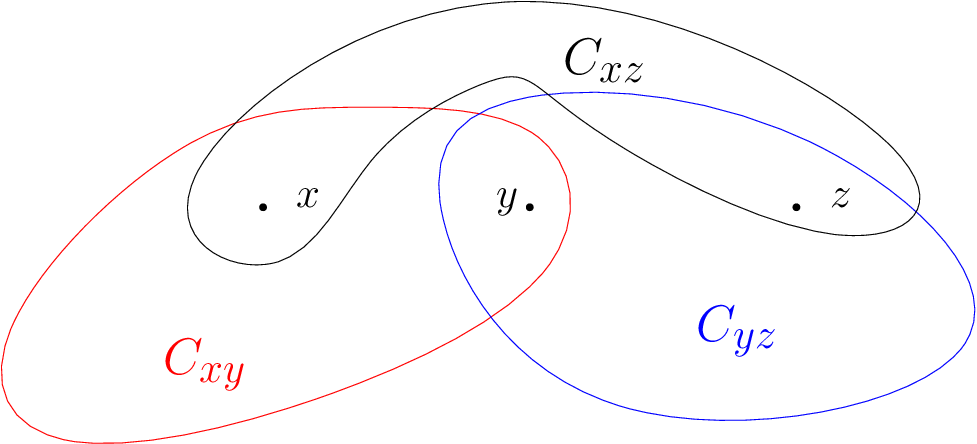}
	
	\caption{The DAG for random walls; $\CC=\{C_{xy}, C_{xz},C_{yz}\}$.}
	\label{fig:walls}
\end{figure} 
\begin{figure}[ht]	
	\centering
	\includegraphics[height=2in]{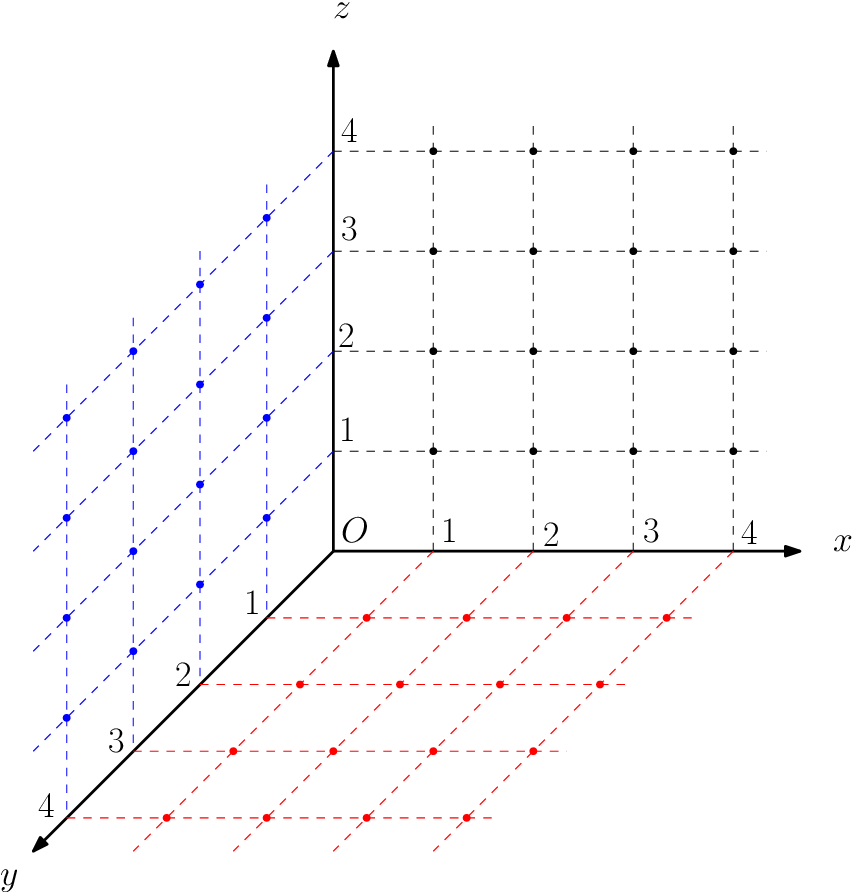}
	\caption{The multi-index set for random walls.}\label{fig:wallpic}
	Axes are not part of the walls, i.e., the walls have no intersections.
\end{figure}
\end{itemize}  
\end{example}




\begin{example}\label{ex:G-auto}
This example further illustrates the notion of $G$-automorphism.
\begin{itemize}
\item[(a)]{\it Nested Sequences:} When $G$ is just a single edge $(v_1 \to v_2)$, the random array with multi-indices in $\mathbb{N}^G$ represents a random sequence whose elements are again sequences. A bijection $\tau$ on the multi-index set $\mathbb{N}^G$ is a $G$-automorphism if and only if it is of the form
$$
        \tau[(v_1,v_2) \mapsto (i,j)] = [(v_1,v_2) \mapsto (\pi(i),\pi'_i(j))]
$$
for some permutations $\pi, \pi'_i$ on $\mathbb{N}$. Here $(v_1,v_2) \mapsto (i,j)$ represents a multi-index in $\mathbb{N}^G$ mapping $v_1$ and $v_2$ to $i$ and $j$, respectively. 
Note the dependency of $\pi'_i$ on the value $i$ of $v_1$. This dependence allows $\tau$ to use different permutations for $v_2$ according to different values of $i$. However, when the edge $\overrightarrow{v_1v_2}$ is removed from $G$, this dependency is no longer allowed: in this case, for a bijection $\tau$ to be a $G$-automorphism, it should have the form
$$
        \tau[(v_1,v_2) \mapsto (i,j)] = [(v_1,v_2) \mapsto (\pi(i),\pi'(j))]
$$
for some permutations $\pi, \pi'$ on $\mathbb{N}$ precisely as in Example \ref{ex:known}(b). Note the use of a single bijection $\pi'$ for all values $i$ of $v_1$, when the edge $\overrightarrow{v_1v_2}$ is removed.

\item[(b)]{\it Sequences of Random Matrices (revisited):} When $G$ is the DAG in {\bf Figure~\ref{fig:DAG:matrix sequence}}, the following bijection $\tau$ on $\mathbb{N}^G$ is a $G$-automorphism: 
$$
        \tau([(s,r,c) \mapsto (i,j,k)]) = [(s,r,c)\mapsto (i,\pi_i(j),k)].
$$
Here $[(s,r,c)\mapsto (i, j, k)]$ is the multi-index in $\mathbb{N}^G$ mapping vertices $(s,r,c)$ to natural numbers $(i,j,k)$, and $\pi_n$ is the permutation on $\mathbb{N}$ that cycles the first $n$ numbers (i.e., $\pi_n(1) = 2$, $\pi_n(2) = 3$, $\ldots$, $\pi_n(n) = 1$, and $\pi_n(m) = m$ for $m > n$). As we previously explained, a random array with indices in $\mathbb{N}^G$ is a sequence of random matrices. The $G$-automorphism $\tau$, in this example, permutes the rows of these matrices, but the way it does so depends on the position of a matrix in the sequence. If we remove the edge $\overrightarrow{sr}$ from $G$, this dependence is no longer permitted, so that $\tau$ stops being a $G$-automorphism. On the other hand, removing the other edge $\overrightarrow{sc}$ from $G$ is harmless; $\tau$ continues to meet the conditions of being a $G$-automorphism.

\item[(c)] {\it Random Block Matrices and Sequences (revisited):}
Consider the middle multi-index set in {\bf Figure \ref{fig:bm-seq1}}. The allowable $G$-automorphisms for this multi-index set are combinations of (a) permuting rows of blocks, (b) permuting columns of blocks, (c) permuting rows within a given block, (d) permutating columns within a given block, and (e) permuting a sequence stored in the entry of a nested matrix.

\item[(d)] {\it $G$-automorphism:} More generally, for any finite DAG $G$, a $G$-automorphism $\tau$ always has the form of applying a permutation to the number associated with each vertex by a given multi-index. The choice of permutation for each vertex is allowed to vary, but only in a way consistent with the structure of $G$. When $\overrightarrow{w_1v},\ldots,\overrightarrow{w_pv}$ are all the incoming edges to a vertex $v$ in $G$ (i.e., $w_1,\ldots, w_p$ is the set of parents of $v$), the permutation for $v$ should have the form $\pi_{(n_1,\ldots,n_p)}$, where the subscripts are the numbers assigned to $w_1,\ldots,w_p$ by a given multi-index $\alpha$, i.e., $\tau(\alpha)(v) = \pi_{(\alpha(w_1),\ldots,\alpha(w_p))}(\alpha(v))$.
\end{itemize}
\end{example}

As remarked in the introduction, our final example illustrates that the multi-index sets of DAG-type random arrays also have a natural infinitely-branching DAG structure.

\begin{example}\label{ex:infdag}
	Ignoring the edges in {\bf Figure \ref{fig:wallpic}} which are merely a visual aid, the vertex set is a natural infinite multi-index set corresponding to the DAG-exchangeable array of {\bf Figure \ref{fig:walls}}. When DAGs have directed edges, the principle still holds: there is a natural infinitely-branching DAG $G'=(V',E')$ underlying the multi-index set of any DAG-exchangeable array for some finite DAG $G=(V,E)$.

        The infinitely-braching DAG $G'$ has a vertex set which replaces each vertex $v$ in $G$ by a countably infinite number of vertices (considered copies of the original $v$). Its edge set is
        	 chosen such that if there is an edge from a copy of $w$ to a copy of $v$ in $G'$, then there is an edge $\overrightarrow{wv}$ in $G$. The precise definition of $G'$ requires a few notations. For each vertex $v \in V$, let 
        \begin{align*}
                V_v & = \{w \in V\,:\,\text{there is a path (possibly 0-length) from $w$ to $v$ in $G$}\},
                &
                \mathbb{I}_v & = \mathbb{N}^{V_v}.
        \end{align*}
        Here, $\mathbb{I}_v$ represents the set of copies of $v$, which together replace the vertex $v$. Thus, $V'=\bigcup_{v\in V} \mathbb{I}_v$. Note that every edge $\overrightarrow{wv}$ in $E$ induces a map from vertices in $\mathbb{I}_v$ to those in $\mathbb{I}_w$ in the graph $G'$. The map transforms a vertex 
$\alpha \in \mathbb{I}_v$ to $\alpha|_{V_w} \in \mathbb{I}_w$, the restriction of $\alpha$ to the sub-domain $V_w$. In $G'$, there is a directed edge from $\alpha \in \mathbb{I}_w$ to $\beta \in \mathbb{I}_v$ if and only if there is a directed edge from $w$ to $v$ in $E$ and the restriction map induced by this edge maps $\beta$ to $\alpha$.

\begin{itemize}
        \item[(a)]{\it $G'$ for Hierarchical Exchangeability:} {The infinite graphs for Austin-Panchenko arrays are collections of infinitely-branching trees. See {\bf Figure~\ref{fig:AP}}. Concretely, consider such an array for the DAG $G$ in {\bf Figure~\ref{fig:DAG:austin-panchenko}}. Let $w_1,\ldots,w_\ell$ be the the terminal vertices of $\ell$ paths in $G$; in the figure, they are labeled by $v^{(1)}_{r_1},\ldots,v^{(\ell)}_{r_\ell}$. The graph $G'$ in this case consists of $\ell$ infinitely-branching trees of depths $r_1,\ldots,r_\ell$, respectively. The multi-indices for the array are tuples $(\alpha_1,\ldots,\alpha_\ell)$ of vertices of $G'$ such that $\alpha_i \in \mathbb{I}_{w_i}$ for $1 \leq i \leq \ell$. They can also be understood as tuples of $\ell$ paths in $G'$, where the $i$-th path starts from the root of the $i$-th tree and repeatedly moves toward the leaves by taking the $v^{(i)}_j$-th child at step $j$ until the path hits a leaf. Also, $\mathbb{I}_{w_i}$ is isomorphic to $\N^{r_i}$ for all $i$ in this case. Thus, the multi-indices just defined are precisely the elements of $\N^{r_1}\times\cdots\times\N^{r_\ell}$, the multi-index set that we have used to describe Austin-Panchenko arrays thus far.}
\begin{figure}[ht]
	\centering
	\includegraphics[height=2.2in]{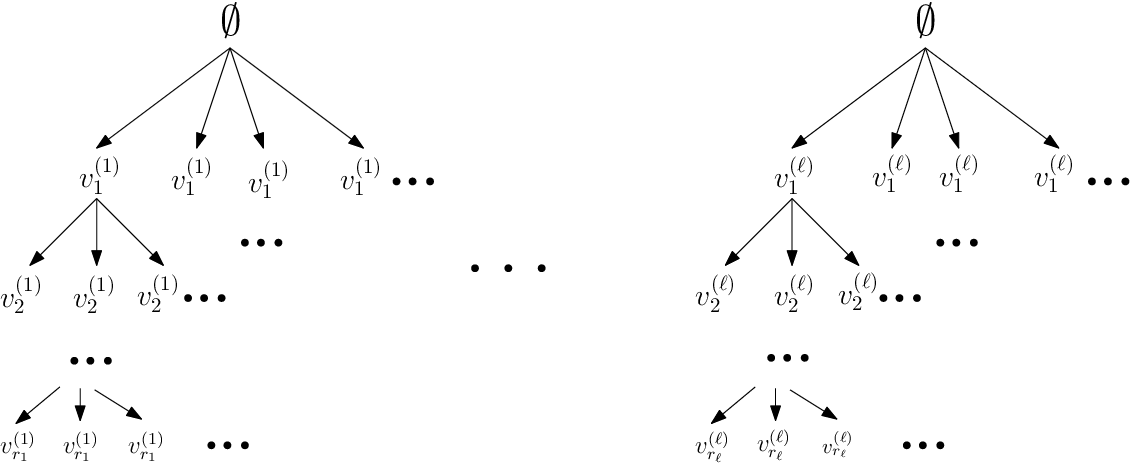}
	\caption{An Austin-Panchenko forest with $\ell$ trees.}
	\label{fig:AP}
\end{figure}

\item[(b)]{\it $G'$ for Random Block Matrices:}
{As shown in {\bf Figure \ref{fig:bmi}}, the $G'$ corresponding to random block matrices has infinitely many copies of $r_0,c_0,r_1$ and $c_1$, respectively. The copies of $r_0$ and $c_0$ correspond to the rows and columns of the outer matrix, and those of $r_1$ and $c_1$ to the rows and columns of the inner nested matrices. The latter copies of $r_1$ and $c_1$ are grouped when they belong to the same nested matrix, and the copies in the same group have incoming edges from one copy of $r_0$ and one copy of $c_0$, which express the position of the block (or inner matrix) within the outer matrix. The multi-indices in this case are pairs of copies of $r_1$ and $c_1$ that belong to the same group.} 
\begin{figure}[ht]
	\centering
	\includegraphics[height=2.5in]{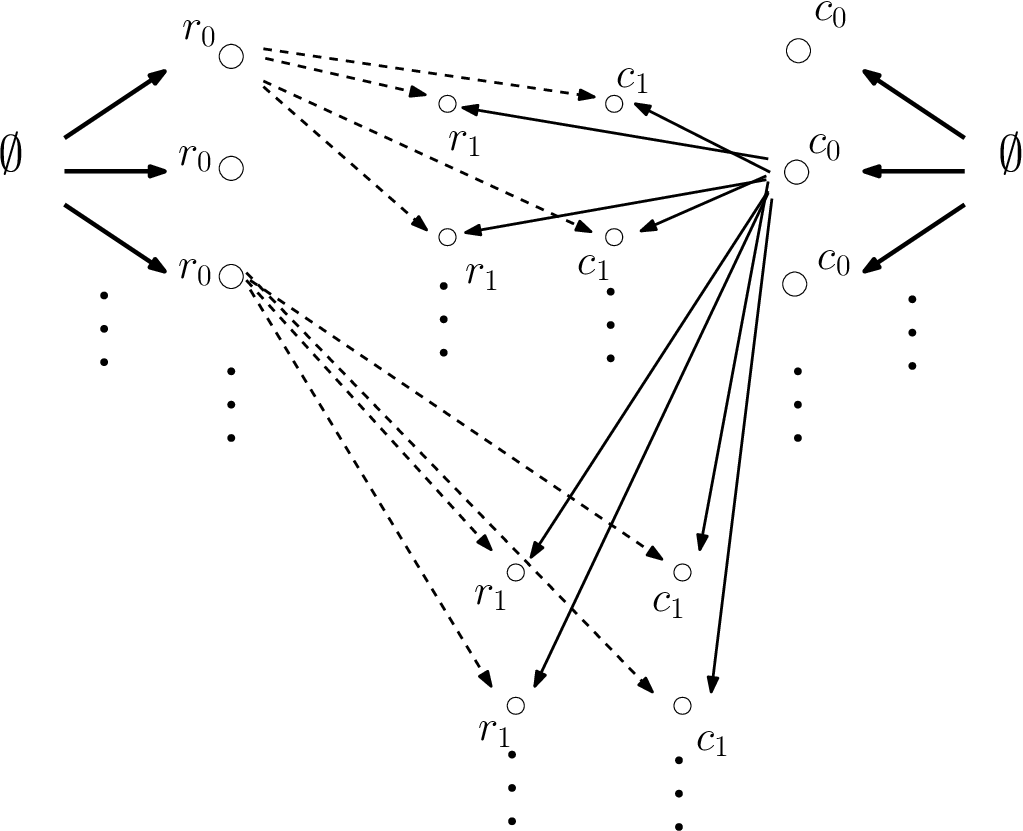}	
	\caption{Part of the infinitely-branching DAG $G'$ corresponding to random block matrices.}
	\label{fig:bmi}
\end{figure}
	\end{itemize}
\end{example}



\hide{
\subsection{Overview of Infinitary DAGs and Their Finite Presentations}\label{sec:setup:overview}

\hy{This section is indeed confusing. The complication comes because we tried to justify the use of DAG in the title of the paper. It may be better just to say that the index set is $\mathbb{N}^G$ for some DAG $G$ in this part of othe paper.}

To begin, we recall the index sets of four well-known exchangeable random arrays, and explain how their elements correspond to subgraphs of infinitary DAGs of finite depth; these subgraphs will serve as indices of the random arrays we will consider. 
{\it Infinitary} here means each vertex, other than a terminal vertex, has infinitely many outgoing edges.
For (a) an exchangeable sequence of de Finetti type, (b) a row-column exchangeable array and (c) a general $\ell$-dimensional exchangeable array of Aldous-Hoover type, the index sets are respectively given by
$$
        \text{(a)}\ \ \N\qquad\qquad\text{(b)}\ \ \N\times\N\qquad\qquad \text{(c)} \ \ 
        \underbrace{\N\times\cdots\times\N}_{\ell \text{ times}}.
$$
In the hierarchical-exchangeability setting of \cite{austin:panchenko:2014}, case (c) is generalized so that each separately exchangeable `dimension' of the array is indexed, not by numbers in $\N$, but rather by sequences of numbers in $\N^r$. The index set in this case is
$$
        \text{(c')} \ \ \N^{r_1}\times\cdots\times\N^{r_\ell}.
$$

All four index sets above are associated with infinitary DAGs, in fact forests or collections of trees,
of finite depth. A given index in each of the four index sets then corresponds to a collection of leaves of the associated DAG (one leaf from each tree in the forest). While it is possible to also view indices of general DAGs as a collection of ``terminal vertices'', for complete generality, it is much more convenient to associate a subgraph to this collection of leaves-- this subgraph will just be a collection of paths (one path from each tree in the forest). For instance, the set $\N$ in (a) is associated with a tree of depth $1$ whose root has infinitely many children, i.e. a single tree in {\bf Figure} 1, and each index $\alpha\in\N$ corresponds to a path from the root to the $\alpha$-th leaf. The index set in (c') is associated with a DAG consisting of $\ell$ infinitary trees of depths $r_1,\ldots,r_\ell$, respectively. The element
$$
\alpha = \Bigg(\left(v^{(1)}_1,\ldots,v^{(1)}_{r_1}\right),\ldots,\left(v^{(\ell)}_1,\ldots,v^{(\ell)}_{r_\ell}\right)\Bigg) \in \N^{r_1}\times\cdots\times\N^{r_\ell}
$$
corresponds to a subgraph made out of $\ell$ paths, where the $i$-th path starts from the root
of the $i$-th tree and repeatedly moves toward the leaves by taking the $v^{(i)}_j$-th child at step $j$
until the path hits a leaf.

\begin{figure}[ht]
	\centering
	\includegraphics[height=2.4in]{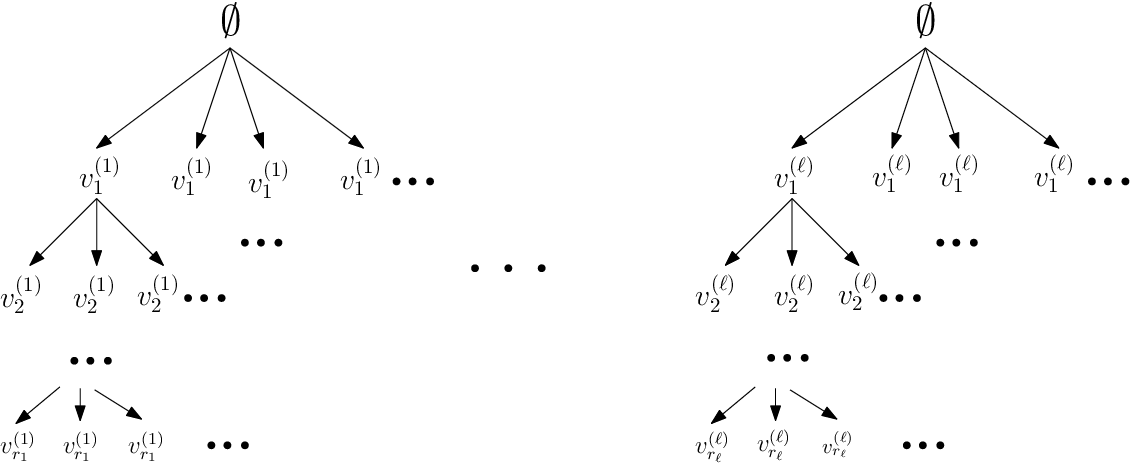}
	\caption{An Austin-Panchenko forest with $\ell$ trees.}
	\label{fig:AP}
\end{figure}
The presence of such an associated DAG and the correspondence between indices and subgraphs is not accidental.  There is a general method for constructing an index set by first building an infinitary DAG of finite depth. The four index sets that we discussed above can all be naturally constructed by this method.

Assume that a \emph{finite} DAG $G$ is given. This $G$ depicts the skeleton of an \emph{infinitary} DAG to be built. In case (a), the finite DAG $G$ is just a single vertex $v$ and no edges; in case (b), $G$ is the DAG with two vertices, $r$ (row) and $c$ (column), and no edges; in case (c'), it is the DAG consisting of $\ell$ disjoint paths of length $r_1,\ldots,r_\ell$ (see {\bf Figure} \ref{fig:DAG:austin-panchenko}).

Next, we generate an infinitary DAG $G'$ from $G$ by making infinitely many copies of vertices of $G$ and connecting these copies by edges in an appropriate manner which we now describe. For every vertex $v$ of $G$, define the {\it downset} of $v$ to be
$$
        D_v \defeq \{w ~:~ \text{there is a path from $w$ to $v$ in $G$}\}.
$$
\tcd{Here a path may be zero-long, and $D_v$ always contains $v$. For instance, if $G$ is just a single vertex $v$ without any edges, $D_v = \{v\}$.} The vertex set of $G'$ is $\bigcup_v \N^{D_v}$ where $\N^{D_v}\defeq\N^{|D_v|}$. That is, each vertex of $G'$ is a function $$\aa:D_v\to \N,$$ for some $v$. The right way to understand such a vertex $\aa$ is as one of the infinitely many copies of $v$ of $G$ that is assigned an identifier $\alpha$. This identifier is then used to connect the various copied vertices with directed edges. In particular, there is a directed edge from $\alpha \in \N^{D_v}$ to $\beta \in \N^{D_w}$ in $G'$ if $(v,w)$ is a directed edge of $G$ and $\beta|_{D_v} = \alpha$. For each of (a), (b), (c) and (c') from above, the associated infinitary DAGs are constructed this way, and they turn out to be forests (see {\bf Figure} \ref{fig:AP}). We will soon see examples where the infinitary DAGs are not forests (see {\bf Figure} \ref{fig:bmi}).

Finally, we set the index set to be $$\{\aa ~: \aa \text{ is a function from the vertex set of }G\text{ to }\N\}.$$
One should view this as a collection of subgraphs of $G'$ in the following way: For each $\alpha : G \to \N$,
 the subgraph of $G'$ associated to $\alpha$ has vertex and edge sets
\begin{align*}
        V'' & = \Big\{\alpha|_{D_v} ~:~ v \in G\Big\},
        &
        E'' & = \Big\{(\alpha,\beta)  ~:~ (\alpha,\beta)\,\text{ is an edge in $G'$}\,
        \text{ and }\,\alpha,\beta \in V'' \Big\}.
\end{align*}

}




\section{Main Result}\label{sec:result}
Let $G$ be a finite DAG and recall that $\AA_G$ denotes the set of all closed subsets of $G$.
By definition, a $G$-automorphism $\tau$ induces a bijection on the $C$-type multi-indices $\bb\in \indices C$ for any $C\in\AA_G$:
$$
\tau(\bb) \defeq  \res{\tau(\aa)}C,\quad \mbox{ for some/any $\aa \in \indices G$ such that $\res \aa C=\bb$.}
$$
Slightly abusing notation, we reuse $\tau$ to denote this induced map. Also, a bijection $\tau\colon \indices C\to \indices C$ acts on a $C$-type random array $\bX_C = (X_\aa : \aa \in \indices C)$ by 
$$
\tau(\bX_C)\defeq {(X_{\tau(\aa)} : \aa \in \indices C)}.
$$ 


\begin{defi}
        Let $\CC$ be a sequence of distinct closed subsets of $G$.
	A $\CC$-type random array collection $\X=(\bX_C : C \in \CC)$ is \textbf{DAG-exchangeable} if it is equal in distribution to $( \tau(\bX_C) : C \in \CC)$ for every $G$-automorphism $\tau$, that is,
	$$ \X \ed (\tau(\bX_C) : C \in \CC).$$
\end{defi}


We denote the set of all multi-indices over some $D \in \AA_C$ by
$$
        I_C  \defeq \bigcup_{D \in \AA_C}\,\N^D.
$$
Also, we introduce the following notations for multi-indices $\aa \in I_G$:
\begin{align*}
        \dom & \colon I_G \to \AA_G,
        &
        \dom(\aa) & \defeq \text{the set of vertices where $\aa$ is defined},
        \\ 
        \rstr & \colon I_G \to 2^{\AA_G},
        &
        \rstr(\aa) & \defeq \{\aa|_C ~:~ C \in \AA_G\},
        \\ 
        \srstr & \colon I_G \to 2^{\AA_G},
        &
        \srstr(\aa) & \defeq \rstr(\aa) \setminus\{\aa\}.
\end{align*}
The $\dom(\aa)$ is the domain of the multi-index $\aa$, and the next two are about
the restrictions of $\aa$: $\rstr(\aa)$ is the set of all the restrictions, while $\srstr(\aa)$ consists
of only the strict restrictions. 

Consider the array $(U_\alpha : \alpha \in I_G)$, where the $U_\alpha$'s are i.i.d. uniform random variables.
Let
$$
\mathbf{V}_\alpha \defeq (U_\beta : \beta \in \rstr(\aa)).
$$
Then, for all $G$-automorphisms $\tau$,
$$
(\mathbf{V}_\alpha : \alpha \in I_G)
\d= 
(\mathbf{V}_{\tau(\alpha)} : \alpha \in I_G)
$$
Thus, for any sequence $\CC$ of distinct closed subsets of $G$ and any family of measurable functions $(f_C : C \in \CC)$ with $f_C : [0,1]^{\AA_C} \to \XX$, 
$$
\Big(\big(f_C(\mathbf{V}_\alpha): \alpha \in \N^C\big) : C \in \CC \Big)
$$ 
is DAG-exchangeable, where for the argument of $f_C$ we identify $U_\beta$ with the $\dom(\bb)$-coordinate of the product space $[0,1]^{\AA_C}$.

As usual, our representation theorem is the converse of the previous statement.

\begin{thm}\label{Thm 1}
If $(\bX_C : C \in \CC)$ is DAG-exchangeable, then
$$
        \Big(\big(X_{C,\aa} : \aa \in \indices C\big) : C \in \CC\Big)
        \d=
        \Bigg(\Big(f_{C}\big(U_\bb : \bb\in \rstr(\aa)\big) : \aa \in \indices C\Big) : C \in \CC\Bigg)
$$
for some family of measurable functions $(f_C : C \in \CC)$ with $f_C : [0,1]^{\AA_C} \to \XX $ and independent $[0,1]$-uniform
random variables $\bU = (U_\aa : \aa \in I_G)$.
\end{thm}

	{\begin{remark}
This is the fine-grained generalization of Theorem \ref{cor 2} alluded to in the introduction. Indeed, Theorem \ref{cor 2} is a simple corollary of the above theorem, since it is just the special case $\CC=\{G\}$.
\end{remark}}

\begin{example}\label{example main} Let us illustrate the application of the above theorem with a very simple example. Consider Example \ref{ex:unknown}(a), which uses the DAG in {\bf Figure \ref{fig:DAG:matrix sequence}}. Let $\mathbb{X}$ be a $\CC$-type random array collection where $\CC=(R,C)$ for $R=\{s,r\}$ and $C=\{s,c\}$. Define
$$
        \bX' \defeq \Big( \big(X_{R,ij}, X_{C,ik}\big) : i,j,k \in \N \Big).
$$ 
This array $\bX'$ is just a way of rewriting $\X$ with straightforward adjustment on indices, and it is easy to see that $\X$ is DAG-exchangeable if and only if 
$$
        \bX' \d= 
        \Big( \big(X_{R,\pi(i)\tau_i(j)}, X_{C,\pi(i)\rho_i(k)}\big) : i,j,k \in \N \Big)
$$ 
for all permutations $\pi, \tau_i, \rho_i \in S_{\N}$ with $i \in \N$.
If $\X$ is DAG-exchangeable, Theorem \ref{Thm 1} tells us that $\bX'$ has a representation of the following form:
$$ 
        \Big( \big(X_{R,ij}, X_{C,ik}\big) : i,j,k \in \N \Big) \d=
        \Big( \big(f_R(U_{000}, U_{i00}, U_{ij0}) , f_C(U_{000}, U_{i00}, U_{i0k}) \big) : i,j,k \in \N \Big)
$$
for some measurable functions $f_R$ and $f_C$ and independent $[0,1]$-uniform random variables $U_{ijk}$ for
$i,j,k \in \{0\} \cup \N$.
\end{example}


The main result of this paper is a probabilistic proof of Theorem~\ref{Thm 1}. Our proof is based on an induction whose inductive step involves reasoning about sophisticated conditional independence, similar to other proofs in the exchangeability literature \cite{kallenberg2006probabilistic,austin:panchenko:2014}. More concretely, in the next subsection, we provide a different version of Theorem~\ref{Thm 1}, from which the theorem follows immediately. Then, we give a  detailed proof of this strengthened version of Theorem~\ref{Thm 1} in Section~\ref{sec:proof}.

\begin{remark}
In Appendix~\ref{app:model-theoretic-proof}, we provide an alternative model-theoretic proof of the special case Theorem~\ref{cor 2} using a result of Crane and Towsner on the representation of relatively exchangeable random structures~\cite{crane2017relative}. Crane and Towsner's result has been formulated and proved in a model-theoretic setting. A large part of our second proof is about translating the graph-theoretic statement of Theorem~\ref{cor 2} to a model-theoretic one in Crane and Towsner's representation theorem, and showing that after translation, the statement satisfies the conditions of Crane and Towsner, and when translated backwards, their conclusion gives the claimed representation of our theorem.
\end{remark}

\subsection{Representations of Random Arrays Induced by Symmetries}\label{sec: symmetry rv S}
We start with a motivation for this section. Consider de Finetti's theorem, the most simple and fundamental result on exchangeability. The theorem is commonly stated as, ``for any exchangeable random sequence, there is a random measure $\mu(\cdot)$ such that conditional on a realization $\mu(\omega)$, the sequence is (conditionally) i.i.d. with common distribution $\mu(\omega)$.'' The theorem was restated by \cite{aldous1981representations} as a representation in terms of uniform random variables: 

\begin{quote}
{\it\noindent For any exchangeable sequence $\bX = (X_n : n \in \N)$ on a Borel space $\XX$, there exists a measurable function $f:[0,1]^2 \to \XX$ such that for an i.i.d. sequence of uniform random variables $\bU=(U_n : n \in \N \cup \{0\}$),} 
\begin{equation}\label{eq: de Finetti} (X_n : n \in \N) \buildrel d \over = \Big(f(U_0, U_n) : n \in \N \Big).
\end{equation}
\end{quote}

Here, $U_0$ takes the role of selecting the random measure $\mu$, while each $U_n$ samples $X_n$ under the law $\mu$. If we choose $\bU$ so that \eqref{eq: de Finetti} is true almost surely instead of in-distribution by  the transfer theorem \cite[Thm 6.10]{kallenberg2002foundations}, one can easily see that $U_0$ is independent of $\bX$ given $\mu$. (See Lemma 7.1 and Theorem 1.1 in \cite{kallenberg2006probabilistic} for example. This type of result is also commonly known as the Hewitt-Savage theorem.)

In this section, we define an array associated to DAG-exchangeable arrays, which we call its \textbf{symmetry array}. It generalizes the role that the empirical random measure $\mu$ plays for exchangeable sequences. We also propose a variant of Theorem \ref{Thm 1}, which states that the uniform random variables in the representation affect the array only through this symmetry array. This corresponds to the fact that $U_0$ is independent of $\bX$ given $\mu$.

Let $\CC$ be a sequence of distinct closed subsets of $G$, and $\X=(\bX_C : C \in \CC)$
be a DAG-exchangeable $\CC$-type random array collection. 

We say that a $G$-automorphism $\tau$ of $\indices G$ \textbf{fixes $\aa\in I_G$} if 
$\tau(\aa)=\aa$. We define $\mathcal{F}_\aa$ to be the sub-$\sigma$-field of $\sigma(\X)$ consisting of $\X$-measurable events that are invariant under every $\aa$-fixing $G$-automorphism~$\tau$:
\begin{align}\label{def:F_alpha}
\begin{array}[t]{@{}r@{}}
\mathcal{F}_\aa \defeq \sigma\Big(\big\{\X^{-1}(B) ~:~ B \text{ is Borel,\, and\, }
\text{if } \big((x_\bb :  \bb\in\indices C) : C \in \CC\big) \in B \text{ and $\tau$ fixes $\aa$, }
\qquad
\\[0.5ex]
        \text{then } \big((x_{\tau(\bb)} : \bb \in \indices C) : C \in \CC\big) \in B\big\}\Big)
\end{array}
\end{align}
One should think of the $\FF_\aa$'s as the $\sigma$-fields which contain the information concerning the symmetries in the array.  In particular,  $\FF_\aa$ contains the information about the symmetries in the array which fix the multi-index $\aa$.

For instance, in Example \ref{example main}, consider $\alpha \in \N^R$ defined by $\alpha(s)=1$ and $\alpha(r)=3$. Then, $X_{R,13}$ is $\FF_\alpha$-measurable, but in general any of $X_{R,23}$, $X_{R,12}$ and $X_{C,13}$ is not. Another example is the empirical distribution $\EE$ of the sequence $(X_{R,13}, X_{C,1k} : k \in \N)$, that is, 
$$
\EE \defeq \lim_{n \to \infty} \frac{1}{n}\, \underset{k \leq n}{\sum} \delta_{(X_{R,13},X_{C,1k})}.
$$ 
By de Finetti-Hewitt Savage theorem, $\EE$ exists almost surely. One can also easily check that it is $\FF_\aa$-measurable. Let
$$
\EE' \defeq \lim_{n \to \infty} \frac{1}{n}\, \underset{k \leq n}{\sum} \delta_{X_{C,1k}}.
$$ 
This empirical distribution is $\FF_\aa$-measurable as well. In fact, it is measurable with
respect to a smaller $\sigma$-field $\FF_{\aa|_{\{s\}}} \subseteq \FF_\aa$.

\begin{remark}
Restricting $\aa$ shrinks the $\sigma$-field $\mathcal{F}_\aa$. That is, for all closed $D \subseteq \dom(\aa)$, we have that $\mathcal{F}_{\aa|_D} \subseteq \mathcal{F}_\aa$. This is because every $\aa$-fixing $\tau$ is also an $\aa|_D$-fixing $G$-automorphism and so an event invariant under the latter kind of $G$-automorphism is also invariant under the former kind.
\end{remark}


It will be convenient to encode into  random variables, the information contained in the various $\FF_\aa$'s. These random variables will later serve as a mechanism by which we make different overlapping representations consistent (in a manner later described).
Using the facts that the elements of each $\bX_C$ take values in a Borel space and that each $\FF_\aa $ is countably generated, we may define:

\begin{defi}
Given a DAG-exchangeable array collection $\X$, we define an associated
 {\bf random symmetry array} $\bS =(S_\aa : \aa \in I_G)$ (after extending the underlying probability space if needed) to be any array $\bS$ satisying
\begin{enumerate}
        \item $\sigma(S_\aa)=\mathcal{F}_\aa $ for all $\aa$, and
	\item the random array collection $(\bS_C : C \in \CC)$ 
	with \mbox{$\bS_C = (S_\aa : \aa \in \indices C)$,} satisfies
	$$
	\big((\bX_C,\bS_C) : C \in \CC\big) \d= \big((\tau(\bX_C),\tau(\bS_C)) : C \in \CC\big)
	$$
	for all $G$-automorphisms $\tau$, i.e. the collection of ordered pairs $((\bX_C,\bS_C) : C \in \CC )$ is DAG-exchangeable.
\end{enumerate}
\end{defi}
To see why such an $\bS$ exists, fix $C_0 \in \CC$ and pick $\alpha_0 \in \N^{C_0}$. There exists a random variable $S_{\alpha_0}$ that generates $\mathcal{F}_{\aa_0}$ (\cite{resnick2013probability}, Ch. 3, Exer. 13). Furthermore, since $S_{\alpha_0}$ is $\X$-measurable, there exists a measurable function $f$ such that $S_{\alpha_0} = f(\bX_C: C \in \CC)$. Now for each $G$-automorphism $\tau$, we can define $S_{\tau(\alpha_0)} \defeq f(\tau(\bX_C): C \in \CC)$. 
Repeating this procedure for each $C \in \CC$ gives $\bS = ((S_\aa : \aa \in \indices C) : C \in \CC)$, which satisfies the two required properties.


For all $\bb \in I_G$, if $\bb$ is a restriction of $\aa$, the random variable $S_{\bb}$ is $\mathcal{F}_\aa$-measurable. This is because $S_{\bb}$ is $\mathcal{F}_\bb$-measurable but the $\sigma$-field $\mathcal{F}_\bb$ is included in $\mathcal{F}_\aa$.

\begin{prop}\label{prop 3} 
	Let $\CC$ be a sequence of distinct closed subsets of $G$ and let $\bS$ be the {symmetry} array defined as above. If $(\bX_C : C \in \CC)$ is DAG-exchangeable, there exist a family of measurable functions $(h_C : C \in \AA_G)$ with $h_C : [0,1]^{\AA_C} \to \XX$ and 
a collection of independent $[0,1]$-uniform random variables $\bU = (U_\aa : \aa \in I_G)$ such that
\begin{equation}\label{equation: S and S'}
        \big(S_\aa : \aa \in I_G\big) \;\d=\; \big(S'_\aa : \aa \in I_G\big)
\end{equation}
where
\begin{equation}\label{equation: generalized representation of DAG} 
        S'_\aa \;\defeq\; h_{\dom(\aa)}\Big(\big(S'_{\bb} : \bb \in \srstr(\aa)\big),\, U_\aa \Big)
\end{equation}
for $\aa \in I_G$.
\end{prop}

Proposition \ref{prop 3} provides a representation for $\bS$ that is built out of a collection of independent random variables $(U_\aa : \aa \in I_G)$ and appropriate measurable functions. {The representation is given in terms of the inductively-defined random variables $(S'_\aa : \aa \in I_G)$, with induction being applied to the size of the domain of each multi-index in $I_G$. Two immediate consequences of the representation are that each $S_\aa$ depends only on $(U_\bb : \bb \in \rstr(\aa))$, and that its dependence on $(U_\bb : \bb \in \srstr(\aa))$ is always mediated via $(S_\beta : \beta \in \srstr(\aa))$.}


We now show that Proposition \ref{prop 3} implies Theorem \ref{Thm 1}. Note that using induction, we can convert $h_{\dom(\aa)}$ to a function $h'_{\dom(\aa)}$ for each $\aa$ such that
\begin{equation}\label{eqn: consequence 1 of representation of DAG} 
        (S_\aa : \aa \in I_G) \;\d=\; \Big(h'_{\dom(\aa)}\big(U_\bb : \bb\in \rstr(\aa)\big) : \aa \in I_G\Big).
\end{equation}
The key part of this inductive conversion is to set $h'_{\dom(\aa)}$ using the following equation:
\begin{multline*}
        h'_{\dom(\aa)}(U_\bb : \bb \in \rstr(\aa))
        = {} \\
        h_{\dom(\aa)}\Bigg(\Big(h'_{\dom(\bb)}\big(U_{\gamma} : \gamma \in \rstr(\bb)\big) : \bb \in \srstr(\aa)\Big),\, U_\aa \Bigg).
\end{multline*}
\begin{proof}[Proof of Theorem \ref{Thm 1}] $X_{D,\aa}$ is $\X$-measurable, and it is fixed under the action of every $\aa$-fixing $G$-automorphism. Thus, $X_{D,\aa}$ is $\FF_\aa$-measurable by the definition of the $\sigma$-field $\FF_\aa$. This means that $X_{D,\aa}$ is also $S_\aa$-measurable because $\sigma(S_\aa) = \FF_\aa$, 
Furthermore, $X_{D,\aa}$ takes values in a Borel space. Thus, there exists a measurable function $f_\aa$ such that $X_{D,\aa} = f_\aa(S_\aa)$ almost surely. 
By the DAG-exchangeability of the collection of ordered pairs $((\bX_C,\bS_C) : C \in \CC)$, we can pick $f_\aa$ such that it depends only on $\dom(\aa)$ and not on the value of $\aa$ itself. This means that we can write $X_{D,\aa} = f_{D}(S_\aa)$ almost surely, by writing $f_\alpha$ as $f_{\dom(\aa)}$. Plugging in \eqref{eqn: consequence 1 of representation of DAG} finishes the proof. 
\end{proof}

\begin{remark} 
        Before getting into the proof of Proposition~\ref{prop 3}, we recall a generic property of exchangeable structures. Whenever $\bX =(X_n : n \in \N)$ is a sequence, by Kolmogorov's extension theorem, its exchangeability is equivalent to the seemingly weaker condition that the distribution of $\bX$ is invariant under the action of \textit{finite} permutations (permutations fixing all but finitely many elements). In particular, if $\bX$ is exchangeable, then $\bX = (X_n : n \in \N) \d= (X_{\tau(n)} : n \in \N) = \tau(\bX)$ for any \textit{injection} $\tau$. (In fact, Ryll-Nardzewski's theorem tells us the converse is also true.) We can extend this sort of argument to other random variables associated to the symmetries of $\bX$.
	
        Let $Y$ be $\bX$-measurable and let $K$ be a subgroup of the infinite permutations. By definition, $Y$ is invariant under the action of $K$ if and only if $(Y, \bX) \d= (Y, \tau(\bX))$ for all $\tau \in K$. By the above paragraph, this is equivalent to having $(Y, X_n : n \in F) \d= (Y, X_{\tau(n)}: n \in F)$ for all $\tau \in K$ and all \emph{finite} $F \subseteq \N$. Moreover, we have  $(Y,\bX)\d=(Y,\rho(\bX))$ for any injection $\rho$ on $\N$ such that its arbitrary restriction to finite sets can be extended to an element in $K$. 
        

\end{remark}

\begin{table}[t]
	\begin{center}
		\caption{Notation Guide}
		\label{tab:table1}
		\begin{tabular}{l|r} 
			Symbol & Object \\
			\hline
			$G$& finite DAG or vertex set of a finite DAG \\
			$C$& downward-closed (w.r.t. partial ordering) subset of $G$ \\
			$W, H$& arbitrary subsets of $G$ \\
			$\AA_G$& set of all downward-closed subsets of $G$\\
			$\CC$& sequence of {distinct}, downward-closed subsets of $G$ (i.e. $\CC\subset\AA_G$)\\ 
			$\N^G, \N^C, \N^H$ & index sets corresponding to vertex sets of $G,C,H$\\
			$I_C$&  the multi-graph index set: $\bigcup_{D \in \AA_C}\,\N^D$\\ 
			$\alpha, \beta, \gamma$ & elements of some index set \\
			$\tau, \rho$ & $G$-automorphisms \\
		\end{tabular}
	\end{center}
\end{table}


\section{Proof of the Main Result}\label{sec:proof}
To prove our main result, it remains to prove Proposition~\ref{prop 3}. Let $G$, $\CC$ and $\X$ be a finite DAG, a sequence of distinct closed subsets of $G$, and a $\CC$-type DAG-exchangeable random array collection from the proposition. Also let $\bS$ be {the symmetry array} defined as in Section \ref{sec: symmetry rv S}. 

\vspace{3mm}

\noindent {\bf Overview of the proof of Proposition~\ref{prop 3}}
\vspace{3mm}

{As typical for probabilistic proofs in the exchangeability literature, our proof is by induction. Before describing an overview of our proof, it is pedagogical to introduce a natural alternative approach which is also an induction; the difficulty in realizing this alternative approach helps clarify what we believe to be the `crux' of proving a  representation theorem for DAG-exchangeable arrays, and is what eventually guides us in how to
   organize the actual proof.} The alternative approach is first to apply the inductive hypothesis to every $C \in \AA_G \setminus \{G\}$, i.e. assume a representation \eqref{equation: S and S'} for $(S_\alpha : \alpha \in I_C)$ for every such $C$, and then to prove the inductive step by combining these representations to get \eqref{equation: S and S'} for the entirety of $G$. Although this seems natural, this approach is difficult to implement. This is because different $C$ and $C'$ in $\AA_G \setminus \{G\}$ may still share vertices (i.e., $C \cap C' \neq \emptyset$) and the representations obtained by applying the inductive hypothesis to $C$ and $C'$ may differ on those vertices; these representations may induce different representations for $(S_\alpha : \alpha \in I_{C \cap C'})$. We will henceforth say that these two representations are  {\bf consistent} if they are the same for all $\alpha \in I_{C \cap C'}$

{The architecture of our induction is built to overcome the above-described difficulty, i.e., built to make representations corresponding to $C$ and $C'$ consistent on $C\cap C'$. In fact, there are two levels of induction in our proof. At the top level, there is a rather simple induction on the number of vertices $n$ of $G$.
	The top-level inductive assumption allows us to assume representations for all closed subsets with less than $n$ vertices. The more difficult second-level induction is designed to make consistent, in a systematic way, the possibly different representations for all the different closed subsets having less than $n$ vertices.
		The base step $k=0$ of our second-level induction, is to choose from any one of the closed subsets of size $n-1$, a representation for the array $(S_\alpha : \alpha \in I_{G_0})$ for the closed $G_0$ defined by
\begin{equation}
        \label{eqn:notation:G0}
        G_0\defeq G\setminus T.
\end{equation}
Here $T$ is the set of all {\it terminal} vertices (i.e. vertices with no descendants). Our second-level induction is an induction on the size of sets $A\subset T$ that we will now add to $G_0$. More specifically, the inductive step of our second-level induction is to show that, whenever consistent representations for arrays of the form $(S_\alpha : \alpha \in I_{G_0\cup A})$ exist for $A$'s such that $|A|=k-1$, and potentially non-consistent representations exist also for $|A|=k$, then one can appropriately combine the consistent representations at level $k-1$ to obtain consistent representations at level $k$.}


The second-level induction gives us consistent representations \eqref{equation: S and S'} for arrays of the form  $$(S_\alpha : \alpha \in I_{G_0 \cup A} \text{ for some } A \subseteq T, |A|=k)$$ for all $k$ such that $|G_0|+k<|G|$.
	To finish the proof, we must complete the top-level induction by extending the consistent representations at level $n-1$ (obtained via the second-level induction), which give a joint representation of the array $(S_\alpha : \alpha \in I_G \backslash \N^G)$,  to the whole array $(S_\aa : \alpha \in I_G)$. This final step is easily obtained by an application of an elementary coding lemma (Lemma \ref{lem: coding lemma}).

Now, for each $C \in \AA_G$, set 
\begin{equation*}
        \mathcal{F}_C\defeq \sigma(\bS_C).
\end{equation*}
Clearly, we have $\mathcal{F}_D \subseteq \mathcal{F}_C$ whenever $D \in \AA_C$. 
The key to carrying out the above described induction is the following proposition. We will prove this in Section \ref{subsection: proof}.

\begin{prop}\label{prop: conditional independence on disjoint graphs, 1} Let $C, C_1,\ldots,C_m$ be closed subsets of $G$. 
Then, $$\FF_C \underset{(\FF_{C \cap C_i})_{i \leq m}}{\independent} (\FF_{C_i})_{i\leq m}.$$
\end{prop}

{The proposition will be utilized via the following two immediate corollaries.} Let the set of all terminal vertices of $G$ be denoted by 
\begin{equation}
        \label{eqn:notation:T-t}
        T=\{v_1,\ldots,v_t\},\ \ \text{ where }|T|=t.
\end{equation}
Define
\begin{equation}
        \label{eqn:notation:G0-GGA}
        G_0 \defeq G \backslash T
        \quad\text{ and }\quad
        \GG_A \defeq \FF_{G_0 \cup A}
        \ \text{for $A \subseteq T$}.
\end{equation}
\begin{cor}\label{cor: conditional independence on disjoint graphs, 2} Let $B, B_1,\ldots,B_m \subseteq T$. Then, $$\GG_B \underset{(\GG_{B \cap B_i})_{i \leq m}}{\independent} (\GG_{B_i})_{i\leq m}.$$
\end{cor}
\begin{proof} 
        Apply Proposition~\ref{prop: conditional independence on disjoint graphs, 1}
        with $C = G_0 \cup B$ and $C_i=G_0 \cup B_i$. 
\end{proof}

\begin{cor}\label{prop:inductive joining of distributional equalities} For $k=0,\ldots,t$, let $$\HH_k \defeq \{\GG_A : A\subseteq T, |A|=k\}\quad\text{ and }\quad\GG_k \defeq \sigma(\HH_k).$$ Then, given $\GG_{k-1}$, the set $\HH_k$ is a {family of  independent $\sigma$-fields} for all $1 \leq k \leq t$.
\end{cor}
\begin{proof} When $k=t$, the result is immediate because $\HH_k$ is a singleton set. Assume that $k < t$. Let $A \subseteq T$ be such that $|A|=k$. Set 
$$
        \GG_{k\setminus A}\defeq\sigma(\{\GG_B : B \subseteq T, |B|=k, B \neq A\})
$$ 
and 
$$
        \GG'\defeq\sigma(\{\GG_{B \cap A}: B \subseteq T, |B|=k, B\neq A\}).
$$
By Corollary \ref{cor: conditional independence on disjoint graphs, 2}, we have $\GG_A \underset{\GG'}{\independent} \GG_{k\setminus A}$.
But 
$$
        \GG' \subseteq \GG_{k-1} \subseteq \sigma(\GG'\cup \GG_{k\setminus A}).
$$
This is because $\GG_{B \cap A} \subseteq \GG_B$ for every $B \subseteq T$ (which itself follows from the fact that $\FF_D \subseteq \FF_C$ for closed $D,C$ whenever $D \subseteq C$). Thus,
we have that $\GG_A \underset{\GG_{k-1}}{\independent} \GG_{k\setminus A}$, from which the result follows.
\end{proof}

For any subset $I$ of $I_G$ and random array $\bY = (Y_\aa : \aa \in I_G)$, let us denote the sub-array $$\bY_I\defeq(Y_\aa : \aa \in I).$$

\begin{proof}[Proof of Proposition \ref{prop 3}] 
	Without loss of generality we will assume that $\CC$ is the set $\AA_G$ with some fixed ordering.
        We use two levels of induction in the proof. The top-level induction is on the number of vertices of $G$ where the $n=1$ case is simply the de Finetti-Hewitt-Savage theorem. Using the inductive hypothesis for the $n-1$ case, assume representations exist for 
        \be\label{firstinduction}
        (S_\alpha : \alpha \in I_{G_0 \cup A} \text{ for some } A \subseteq T, |A|=k)
        \ee
        whenever $|G_0|+k<|G|$. Our first objective is to show that such representations can be chosen to be consistent in the sense described at the beginning of this section.
        	
        	Set $|T|=t$. The case where $|T|=1$ is obtained directly from the inductive hypothesis and Lemma~\ref{lem: final piece}, below. In the rest of the proof, we assume that $|T| = t >1$.  

Choose a closed subset of $G$ with $n-1$ vertices. By the fact that it is closed, it must contain $G_0$. By the assumption of \eqref{firstinduction}, there exist Borel functions $\{g_C: C \in \AA_{G_0}\}$ as well as an array of independent $[0,1]$-uniform random variables, $\bU$ (which we can assume to be independent from all of the symmetry arrays below), such that
\begin{align}\label{equation: representation on 0} 
        (S_\aa : \aa \in I_{G_0}) 
        & \d= (S'_\aa : \aa \in I_{G_0}), \nonumber 
        \\ 
        S'_\aa 
        & \defeq g_{\dom(\aa)}(\bS'_{\srstr(\aa)}, U_\aa )\qquad \text{for } \aa \in I_{G_0}.
\end{align}
Similarly, for each $A\subsetneq T$ there exist  $\{g^A_C: C \in \AA_{G_0 \cup A}\}$ such that
\begin{align}\label{equation: representation on 1 marginal} 
        (S_\aa : \aa \in I_{G_0 \cup A}) 
        & \d= (S^A_\aa : \aa \in I_{G_0 \cup A}), \nonumber 
        \\ 
        S^A_\aa 
        & \defeq g^A_{\dom(\aa)}(\bS^A_{\srstr(\aa)}, U_\aa ) \qquad\text{for } \aa \in I_{G_0 \cup A}.        
\end{align}
Here, the arrays $\bS'_{\srstr(\aa)}$ and $\bS^A_{\srstr(\aa)}$ are defined recursively through \eqref{equation: representation on 0} and \eqref{equation: representation on 1 marginal}. We must next show that the above representations can be chosen to be consistent, to which end we use another (second-level) induction on the sizes of the $A$'s, say $|A|=k$.

It is pedagogical to go through the easiest step of induction, from $k=0$ to $1$, before dealing with the general inductive step. To simplify notation, let $A=\{v_s\}$ and set
        \begin{align}\label{eq: v_s}
        S^s_\aa 
 \defeq g^s_{\dom(\aa)}(\bS^s_{\srstr(\aa)}, U_\aa ) \qquad\text{for } \aa \in I_{G_0 \cup \{v_s\}}
 \end{align}
and
\begin{align*}
        \bm\eta_s & \defeq\bS_{I_{G_0 \cup \{v_s\}}},
        &
        \bm\theta_s & \defeq\bS^s_{I_{G_0 \cup \{v_s\}}},
        &
        \bU_s & \defeq\bU_{I_{{G_0} \cup \{v_s\}} \setminus I_{G_0}}.
\end{align*}
Using the second equation of \eqref{eq: v_s} for each $\alpha \in I_{G_0 \cup \{v_s\}}$, we can express each $S^s_\alpha$  in terms of the $S^s_\beta$'s 
and $U_\gamma$'s with $\beta \in I_{G_0}$ and $\gamma \in I_{G_0 \cup \{v_s\}}\setminus I_{G_0}$. The resulting equations can be written as $$\bm\theta_s=F_s(\bS^s_{I_{G_0}}, \bU_s)$$ for an appropriate measurable $F_s$. By Corollary \ref{prop:inductive joining of distributional equalities}, $\underset{\bS_{I_{G_0}}}{\independent} (\bm\eta_s)_s $. By construction, \mbox{$ \bS_{I_{G_0}} \d= \bS'_{I_{G_0}} \d= \bS^s_{I_{G_0}}$} are independent from $(\bU_s)_s$, and $( \bS_{I_{G_0}},\bm\eta_s) \d= (\bS^s_{I_{G_0}},\bm\theta_s)$. Therefore, by Lemma \ref{lem:joining}, we have $$( \bS_{I_{G_0}}, \bm\eta_s)_s \d= (\bS'_{I_{G_0}},F_s(\bS'_{I_{G_0}},\bU_s))_s.$$
Thus we can join the representations given by \eqref{equation: representation on 0} and \eqref{equation: representation on 1 marginal} to obtain the following joint distributional equality
\begin{align}\label{equation: representation on 1 joint}
        (S_\aa : \aa \in I_{G_0 \cup \{v_s\}}, s \leq t) & \d= (S^1_\aa : \aa \in I_{G_0 \cup \{v_s\}}, s \leq t), 
\end{align}
where $S^1_\aa =S'_\aa $ if $\aa \in I_{G_0}$ and the rest of the $S^1_\aa $'s are defined by the recursive formulae $$S^1_\aa =g^s_{\dom(\aa)}(\bS^1_{\srstr(\aa)}, U_\aa ).$$
This is a (consistent) joint representation of $$(S_\aa : \aa \in I_{G_0 \cup A}, A \subseteq T, |A|=k),$$ in the case where $k=1$. 

Let us now generalize the above by carrying out the second-level inductive step on general $k$ to achieve an analogous joint representation at the level $k=t-1$. Now set $k<t-1$ to be fixed and assume that we have the following joint representation:
\begin{align}\label{equation: representation on k joint}\nn
        (S_\aa :\aa \in I_{G_0 \cup A}, A \subseteq T, |A|=k) 
        & \d= (S^k_\aa : \aa \in I_{G_0 \cup A}, A \subseteq T, |A|=k)
        \\
        S^k_\aa 
        & \defeq g_{\dom(\aa)}(\bS^k_{\srstr(\aa)}, U_\aa )\qquad\text{for } \aa \in \underset{|A|=k}{\bigcup}I_{G_0 \cup A}.
\end{align}


 Consider the representation in \eqref{equation: representation on 1 marginal} for any fixed $B \subseteq T$ with $|B|=k+1$ (note that $k+1\le t-1$), and with Borel measurable functions $(g^B_C:C \in \AA_{G_0 \cup B})$. We rewrite it here for convenience: 
\begin{align}\label{equation: representation on k+1 marginal}
        (S_\aa : \aa \in I_{G_0 \cup B}) 
        & \d= (S^B_\aa : \aa \in I_{G_0 \cup B}), \nonumber 
        \\
        S^B_\aa 
        & \defeq g^B_{\dom(\aa)}(\bS^B_{\srstr(\aa)}, U_\aa )\qquad\text{for }\aa \in I_{G_0 \cup B}.
\end{align}

Define the following arrays (to ease notation we do not use boldface for these):
\begin{align*}
        \eta_B   & \defeq (S_\aa : \aa \in I_{G_0 \cup A}, A \subseteq B, |A|=k),
        \\ 
        \theta_B' & \defeq (S^k_\aa : \aa \in I_{G_0 \cup A}, A \subseteq B, |A|=k),
        &\theta_B^B& \defeq (S^B_\aa : \aa \in I_{G_0 \cup A}, A \subseteq B, |A|=k),
        \\
        \\
        U_B & \defeq(U_\aa : \aa \in I_{G_0 \cup A}, A \subseteq B, |A|=k),\\
        U_{B^c} & \defeq (U_\aa : \aa \in I_{G_0 \cup A}, A \subseteq T, |A|=k) \setminus U_B,
        &\partial U_B & \defeq (U_\aa :\aa \in I_{G_0 \cup B}) \setminus U_B,
        \\
        \\
        \eta_T & \defeq (S_\aa :\aa \in I_{G_0 \cup A}, A \subseteq T, |A|=k),
        & 
        \bar\eta_B & \defeq (S_\aa :\aa \in I_{G_0 \cup B}),
        \\
        \theta_T' & \defeq (S^k_\aa : \aa \in I_{G_0 \cup A}, A \subseteq T, |A|=k),
        & 
        \bar\theta^B_B & \defeq(S^B_\aa : \aa \in I_{G_0 \cup B}).
\end{align*}

\noindent where $\setminus U_B$ denotes deletion of the array $U_B$. Then, by Corollary \ref{cor: conditional independence on disjoint graphs, 2}, $\eta_T$ and $\bar\eta_B$ are conditionally independent given $\eta_B  $, and by construction $\eta_B   \d= \theta_B' \d= \theta_B^B$, all of them independent from $U_{B^c}$ and $\partial U_B$. Also, by the first lines of \eqref{equation: representation on k joint} and \eqref{equation: representation on k+1 marginal}, we have $\eta_T\ed \theta_T'$ and $\bar\eta_B\ed \bar\theta^B_B$. Finally, using the second line of \eqref{equation: representation on k+1 marginal} for each $\aa \in I_{G_0 \cup B}$, we can express
$S^B_\aa$ in terms of the $S^B_\bb$'s in $\theta_B^B$ and $U_\alpha$ in $\partial U_B$. Thus, for an appropriate $F_B$, we have 
$$
\bar\theta^B_B=F_B(\theta_B^B, \partial U_B).
$$
Also, by similar reasoning using the second equation in \eqref{equation: representation on k joint}, we get
$$
\theta_T'=F(\theta_B', U_{B^c})
$$
for some $F$.

Now,
by Lemma \ref{lem:joining}, 
$$(\eta_B  ,\eta_T,\bar\eta_B) \d= (\theta_B', F(\theta_B',U_{B^c}), F_B(\theta_B',\partial U_B))$$
and in particular $$(\eta_T, \bar\eta_B) \d= (F(\theta_B',U_{B^c}), F_B(\theta_B',\partial U_B))=(\theta_T', F_B(\theta_B',\partial U_B)).$$
Moreover, we have $\underset{\eta_T}{\independent}(\bar\eta_B)_B$ by Corollary \ref{prop:inductive joining of distributional equalities}, and also $\underset{\theta_T'}{\independent}(F_B(\theta_B', \partial U_B))_B$ since $F_B(\theta_B',\partial U_B)$ is a function of $(\theta_T', \partial U_B \setminus U_{B^c})$, while $(\partial U_B \setminus U_{B^c})_B$ is an independent family which is also independent from $\theta_T'$. Thus, a slight variation of Lemma \ref{lem:Identification} shows that $$(\eta_T, \bar\eta_B)_B \d= (\theta_T', F_B(\theta_B', \partial U_B))_B.$$ Therefore, we have
\begin{align}\label{equation: representation on k+1 joint}
        (S_\aa :{\aa \in I_{G_0 \cup B}}, B \subseteq T, |B|=k+1) & \d= ({S_\aa ^{k+1}}: \aa \in I_{G_0 \cup B}, B \subseteq T, |B|=k+1)
\end{align}
where ${S_\aa ^{k+1}}=S^k_\aa $ for $\aa \in I_{G_0 \cup A}, A \subseteq T, |A|=k$ and for other $\aa \in \indices {G_0 \cup B}$, the random variable ${S_\aa ^{k+1}}$ is defined through the recursive formulae $${S_\aa ^{k+1}}=g^B_{\dom(\aa)}({\bS^{k+1}_{\srstr(\aa)}}, U_\aa ).$$ Thus we have built a $(k+1)$-version of \eqref{equation: representation on k joint}. By inducting up to the level $k=t-1$ (our top-level inductive hypothesis at level $n-1$ only allows us to go this far), we obtain the following representation, which involves everything except for the $S_\aa$'s for $\aa  \in \N^G$.
\begin{align}\label{equation: representation on t-1}
        (S_\aa :{\aa \in I_G\setminus \N^G}) 
        & \d= (S'_\aa : {\aa \in I_G\setminus \N^G}), \nonumber 
        \\
        S'_\aa & \defeq g_{\dom(\aa)}(\bS'_{\srstr(\aa)}, U_\aa )
        \qquad\text{for } \aa \in { I_G\setminus \N^G}.
\end{align}

We have joined all the representations on the proper sub-DAGs. Lemma \ref{lem: final piece}, below, is the final piece of the puzzle to complete the top-level induction to get a representation for the whole array $(S_\aa : \aa \in I_G)$. Using this lemma (defining $\t\bS$ as in the lemma, and defining $\t\bS'$ similarly), we can complete the proof by showing
$$
\Big(\t\bS, \big(f_G(\bS_{\srstr(\aa)}, U_\aa):\aa \in \N^G\big)\Big)
\d= 
\Big(\t\bS', \big(f_G(\bS'_{\srstr(\aa)}, U_\aa):\aa \in \N^G\big)\Big).
$$
This distributional equality holds because
$(\t\bS,(U_\aa:\aa\in \N^G)) \d= (\t\bS',(U_\aa:\aa\in \N^G))$, which itself 
follows from Lemma \ref{lem:Identification} instantiated with 
the following data:
\begin{align*}
        \FF & = \text{ the trivial $\sigma$-field},
        &
        I & =\{1,2\}, 
        \\
        (T_1,T_2) & = (\t\bS,(U_\aa:\aa\in \N^G)), 
        &
        (V_1,V_2) & =(\t\bS',(U_\aa:\aa\in \N^G)).
\end{align*}
\end{proof}
\begin{lem}\label{lem: final piece} Let $$\t\bS\defeq(S_\aa :\aa \in I_G \setminus \N^G).$$ Then, there exists a Borel measurable function $f_G$ such that for any array of independent $[0,1]$-uniform random variables $(U_\aa:  \aa\in \N^G)$ which is independent from $\bS$,
	\begin{equation}\label{equation: final joint}
	\Big(\t\bS, (S_\aa:\aa \in \N^G)\Big) \d= \Big(\t\bS, \big(f_G(\bS_{\srstr(\aa)}, U_\aa):\aa \in \N^G\big)\Big).
	\end{equation}
\end{lem}

This lemma is  a consequence of Lemma \ref{lem: coding lemma} and will be proved in the following subsection.

\subsection{Proofs of Proposition \ref{prop: conditional independence on disjoint graphs, 1} and Lemma \ref{lem: final piece}}\label{subsection: proof}
{Recall the following three objects:
the finite DAG $G$, the sequence $\CC$ of distinct closed subsets of $G$, and the $\CC$-type DAG-exchangeable random array collection $\X$. As in the proof of Proposition~\ref{prop 3}, we will assume that $\CC$ is the set $\AA_G$ with some fixed ordering.}

Let us say that a $G$-automorphism $\tau$ is \textbf{separated} if for all $v \in G$, there exists a permutation $\tau_v$ on $\N$ such that $\tau(\bb)(v)=\tau_v(\bb(v))$ for all $\bb \in \N^G$. The term comes from the fact that an array is separately exchangeable if and only if its distribution is invariant under the action of every separated $G$-automorphism. 
Note that every DAG-exchangeable array collection, including our $\X$, is automatically separately exchangeable.
Therefore, for each (not necessarily closed) subset $H \subseteq G$ and multi-index $\aa \in \N^H$, we may define $\FF_\aa^{\sep}$ to  be the $\sigma$-field of all events which are invariant under the actions of all separated $\aa$-fixing $G$-automorphisms.
\begin{equation*}
\mathcal{F}_\aa^\sep \defeq \sigma\Big(\big\{\X^{-1}(B) ~:~ 
\begin{array}[t]{@{}l@{}}
        B \text{ is Borel, \,and\, } 
        \text{if } \big((x_{\bb} : \bb \in \N^C) : C \in \CC \big) \in B, 
        \\[0.5ex]
        \quad
        \text{then } \big((x_{\tau(\bb)} : \bb \in \N^C) : C \in \CC \big) \in B
        \text{ for all separated $\aa$-fixing $\tau$}\big\}\Big).
\end{array}
\end{equation*}
It is important to remember that the domain $H\subset G$ of $\aa \in \N^H$ here is not necessarily closed.
We will use the letter $H$ below to denote such a general subset of $G$, while continuing our convention that $C$ and $D$
denote closed subsets. Since a Borel set $B$ on the right-hand side above has less $G$-automorphisms it has to be invariant with respect to, compared to the definition of $\FF_\aa$ in \eqref{def:F_alpha}, it follows that $\FF_\aa\subset \mathcal{F}_\aa^\sep$.

The missing ingredient, common to the proofs of both Proposition \ref{prop: conditional independence on disjoint graphs, 1} and Lemma \ref{lem: final piece}, 
is the following conditional independence result which appears as Corollary 5.6 in the celebrated paper of Hoover \cite{hoover1979relations}:
\begin{prop}\label{prop: Hoover conditional independence} Define $\FF_\aa ^{\sep}$ as above. Let $I_1, I_2, I_3 \subseteq \bigcup_{H \subseteq G} \N^H$ be such that, for all $\aa_1 \in I_1$ and $\aa_2 \in I_2$, we have $\aa_1 \cap \aa_2 \in I_3$. Then, $(\FF_\aa ^{\sep}: \aa \in I_1)$ is {conditionally} independent from $(\FF_\aa ^{\sep}: \aa \in I_2)$ given $(\FF_\aa ^{\sep}: \aa \in I_3)$.
\end{prop}
\noindent{In the proposition, $\aa_1 \cap \aa_2$ means the restriction of $\aa_1$ to the set of vertices that get mapped to the same values by $\aa_1$ and $\aa_2$.}

\vspace{3mm}

\noindent{\bf Overview of the proof of Proposition~\ref{prop: conditional independence on disjoint graphs, 1}}
\vspace{3mm}

First note that if we replaced $\FF_C$ (respectively for the $C_i$ and $C\cap C_i$) in Proposition~\ref{prop: conditional independence on disjoint graphs, 1} by 
	\be\label{sepsigma}
	\sigma(\{\FF_\aa ^{\sep}:\aa\in I\}),\qquad I=\bigcup_{H \subseteq C} \N^H
	\ee
	 (respectively for $C_i$ and $C\cap C_i$), then the result would immediately follow from Proposition \ref{prop: Hoover conditional independence}.
While the $\sigma$-fields related to DAG-exchangeability in Proposition~\ref{prop: conditional independence on disjoint graphs, 1} are not of the type in \eqref{sepsigma}, using the structure of $I_C$ and the fact that $C$ is closed, it is possible to express $\FF_C$ as 
\be \label{sep dag rel}
\FF_C=\sigma(\{\FF_\aa ^{\sep}: \aa \in I_C\})
\ee 
(similarly for $C_i$ and $C\cap C_i$). This is established in Lemma~\ref{lem: symmetric variables with respect to X} below. A final order of business required to employ Proposition \ref{prop: Hoover conditional independence} is a sort of converse: we also need to express \eqref{sepsigma} in terms of $\sigma$-fields related to DAG-exchangeability since these are the $\sigma$-fields that one conditions on in Proposition \ref{prop: Hoover conditional independence}. This will be done in Lemma \ref{lem:separate symmetry is null}.

\vspace{3mm}

For $\aa \in \N^G$, define
$$\bX_\alpha \defeq (X_{C,\aa|_C}: C \in \CC)$$
so that henceforth $$\bX\defeq (\bX_\aa:\aa \in \N^G),$$ which we view as an array of arrays.
{Using this notation, we may rewrite \eqref{def:F_alpha} as
\begin{equation}\label{eq: Faa2}
\mathcal{F}_\aa = \sigma\Big(\big\{\bX^{-1}(B) ~:~ 
\begin{array}[t]{@{}l@{}}
        B \text{ is Borel,\, and\, } 
        \text{if } \big(x_{\bb} :  {\bb}\in \N^G \big) \in B, 
        \\[0.5ex]
        \quad
        \text{then } \big(x_{\tau(\bb)} : \bb \in \N^G \big) \in B
        \text{ for all $\aa$-fixing $\tau$}
        \big\}\Big)
\end{array}
\end{equation}}
Let us point out that an $\aa$-fixing $G$-automorphism $\tau$, {acting on $\bX$,} fixes $\aa|_C$ for {$C \in \AA_{\dom(\aa)}$,} but {does not necessarily fix $\aa|_H$ for any arbitrary  (non-closed) subset $H$.} We assume that $\FF_\aa^\sep$ is rewritten similarly as a $\sigma$-field defined in terms of $\bX$, instead of $\X$.


For a subset $H\subseteq G$, not necessarily   closed, let $H^o$ denote the largest subset of $H$ which is   closed in $G$. The closed graph $H^o$ is well-defined since a union of   closed subsets is again   closed. For example, consider the case that the vertex and edge sets of $G$ are $\{v_1,v_2,v_3,v_4\}$ and $\{\overrightarrow{v_1v_2}, \overrightarrow{v_2v_3}, \overrightarrow{v_3v_4}\}$, respectively. If $H_1=\{v_1,v_3\}$ and $H_2=\{v_2,v_3,v_4\}$, then $H_1^o=\{v_1\}$ and $H_2^o=\emptyset$.

The \textit{closure} of a subset $H \subseteq G$ is the smallest closed subset containing $H$.

\begin{lem}\label{lem: symmetric variables with respect to X}
Let $\aa \in I_G$. Then, 
$$
        \FF_\aa 
        = \bigcap_{n \geq 1} \FF_\aa^n
        = \bigcap_{n \geq 1} \mathcal{F}_\aa^{\sep,n},
$$ 
where 
$$
        \FF_\aa ^n\defeq\sigma\Big(\big\{\bX_\bb: \bb \in \N^G \text{ \rm and there is } C \,{\in}\, \AA_{\dom(\aa)} \text{ \rm s.t. } \bb|_C=\aa|_C \text{ \rm and } \bb(v)>n \text{ \rm for all } v \,{\notin}\, C\big\}\Big)
$$ 
and 
$$
        \mathcal{F}^{\sep,n}_\aa\defeq\sigma\Big(\big\{\bX_\bb: \bb \in \N^G \text{ \rm and there is } H \,{\subseteq}\, \dom(\aa) \text{ \rm s.t. } \bb|_H=\aa|_H \text{ \rm and } \bb(v)>n \text{ \rm for all } v \,{\notin}\, H\big\}\Big).
$$
In particular, for all $\aa\in I_G$, $$\FF_\aa ^{\sep}=\FF_\aa .$$
\end{lem}
\begin{proof} 
Pick $\alpha \in I_G$. Define an injection $\tau_k: \N \to \N$, where $$\tau_k(k)=k, \quad \tau_k(k-1)=k+1,\quad \text{ and }\quad \tau_k(m)=m+1 \ \text{ otherwise}.$$ Let $\rho_1$ be the injection on $\N^G$ such that $\rho_1(\bb)(v)=\tau_{\aa(v)}(\bb(v))$ if $v \in \dom(\aa)$, and $\rho_1(\bb)(v)=\bb(v)+1$ otherwise. 
        
        The injection $\rho_1$ can be made to act on any $\bX$-measurable $Y$. This is because such $Y$ is equal to $f(\bX)$ almost surely for some measurable $f$ and we can define $\rho_1(Y)$ to be $f(\rho_1(\bX))$.\footnote{$\rho_1(\bX) = (\bX_{\rho_1(\alpha)})_\alpha$.} The choice of $f$ does not matter here for the following reason. The DAG-exchangeability of $\bX$ implies $\rho_1(\bX) \d= \bX$. Thus, $(\bX,f(\bX),g(\bX)) \d= (\rho_1(\bX),f(\rho_1(\bX)),g(\rho_1(\bX)))$ for all measurable $f$ and $g$. This in turn implies that whenever $f(\bX) = g(\bX)$ almost surely, we also have almost sure equality between $f(\rho_1(\bX))$ and $g(\rho_1(\bX))$. 
       
        We can regard each $E \in \FF_\aa$ as an $\bX$-measurable random variable $\mathbf{1}_E$ and apply $\rho_1$ to it. The outcome $\rho_1(E)$ of this application is the same as $E$.  Thus, for all $E \in \FF_\aa$, we have that $\rho^n_1(E) = E$ almost surely for all $n$. Meanwhile, by construction, $\rho_1^n (F) \in \mathcal{F}^{\sep,n}_\aa$ for any event $F$. Thus, we obtain $\FF_\aa \subseteq \mathcal{F}^{\sep,n}_\aa$, which implies the inclusion $\FF_\aa \subseteq \bigcap_{n \geq 1} \mathcal{F}^{\sep,n}_\aa$.

        Now we show the other inclusion. Consider the following two conditions on $G$-automorphisms~$\rho$: 
        \begin{enumerate}
                \item $\rho$ fixes $\aa$; 
                \item $\rho(\bb)(v)=\bb(v)$ for all $\bb$ and { $v$ having some $u \preceq v$ with $\bb(u)>n$}. 
        \end{enumerate}
        In the second condition, we use the partial order $u \preceq v$ introduced earlier, which means that there is a path of length possibly zero from the vertex $u$ to the vertex $v$ in $G$. Let $\TT_n$ be the set of $G$-automorphisms satisfying these two conditions. Then, $\bigcup_{n\in \N}\TT_n$  generates all the finite\footnote{A $G$-automorphism $\tau$ is finite if $\tau(\aa)=\aa$ for all but finitely many $\aa \in \N^G$.} $\aa$-fixing $G$-automorphisms. We claim that $\mathcal{F}^{\sep,n}_\aa$ is invariant under the action of any $\rho\in \TT_n$.

        To see the claim, fix $H \subseteq \dom(\aa)$ and $\bb \in \N^G$ such that $\bb|_H=\aa|_H$ and $\bb(v)>n$ for all $v \not\in H$. Then, $\rho(\bb)(v)=\bb(v)$ for $v \notin H$ or $v \in H^o$; see the remark following \eqref{eq: Faa2}. For $v \in H\setminus H^o$, there exists $u \notin H$ such that $u \prec v$. Otherwise, the closure of $H^o \cup \{v\}$ 
        is in $H$, contradicting the maximality of $H^o$. But then $\bb(u) > n$ by the choice of $\bb$. Thus $\rho(\bb)(v)=\bb(v)$ holds again in this case.

        Combined with the remark {following Proposition \ref{prop 3}}, the claim implies that $\bigcap_{n \geq 1} \mathcal{F}^{\sep,n}_\aa$ is invariant under all $\aa$-fixing $G$-automorphisms. Therefore, $\bigcap_{n \geq 1} \mathcal{F}^{\sep,n}_\aa \subseteq \FF_{\aa}$. Together with our proof for the other inclusion, this gives $\FF_\aa = \bigcap_{n\geq 1} \mathcal{F}^{\sep,n}_\aa$, as desired.

        The equality $\FF_\aa^{\sep}=\bigcap_{n \geq 1} \mathcal{F}^{\sep,n}_\aa$ is a standard fact. It can also be obtained by repeating our argument for $\FF_\aa =\bigcap_{n\geq 1} \mathcal{F}^{\sep,n}_\aa$ {for the DAG which has the same vertex set as $G$, but has no edges}. Thus, if $\dom(\aa)$ is a closed subset (equivalently, $\aa\in I_G$), we have that $\FF_\aa ^{\sep}=\FF_\aa $.

        Since $\FF_\aa ^n \subseteq \mathcal{F}_\aa^{\sep,n}$, to complete the proof, we just need to show that $\FF_\aa \subseteq \bigcap_{n\geq 1} \FF_\aa ^n $. We point out that proving this inclusion is not needed for what we are trying to show in this subsection, namely, Proposition~\ref{prop: conditional independence on disjoint graphs, 1} and Lemma~\ref{lem: final piece}. However, we spell out the proof here since 
$\FF_\aa \subseteq \bigcap_{n\geq 1} \FF_\aa ^n $ is a natural statement which may be useful for future related works.

        Let $n$ be a natural number large enough that $n > \max_{v \in \dom(\aa)}\aa(v)$. Define $\rho_n(\bb)(v)=\bb(v)$ if $v \in \dom(\aa)$ and $\bb(u)=\aa(u)$ for all $u \preceq v$, and $\rho_n(\bb)(v)=\bb(v)+n$ otherwise. Then, $\rho_n$ is an injection and satisfies the following claim.

        \begin{quote} 
                {\sc Claim.} Let $H$ be a subset of $\dom(\aa)$ that is not necessarily closed. Then, for all $\bb \in \indices G$, if $\bb|_H=\aa|_H$ and $\bb(v) > \aa(v)$ for all $v \in \dom(\aa) \setminus H$, we have $\rho_n(\bb)(v)=\bb(v)$ for $v \in H^o$, and $\rho_n(\bb)(v)=\bb(v)+n$ for $v \not\in H^o$.
        \end{quote}

        It is easy to see why the claim holds for $v \in H^o$ or $v \notin H$. If $v \in H \setminus H^o$, there exists $u \in \dom(\aa)\setminus H$ such that $u \prec v$; otherwise, the closure of $H^o \cup \{v\}$ is in $H$, contradicting the maximality of $H^o$. Therefore, $\bb(u) \neq \aa(u)$. This implies that $\rho_n(\bb)(v)=\bb(v)+n$, proving the claim.

        By the claim, $\rho_n (F) \in \FF_\aa^n$ for every $F \in \mathcal{F}_\aa^{\sep,n}$. Meanwhile, $\rho_n(E)=E$ for every $E \in \FF_\aa $. Also, $\FF_\aa \subseteq \mathcal{F}_\aa^{\sep,n}$. Thus, $\FF_\aa \subseteq \FF_\aa ^n$.
\end{proof}


\begin{lem}\label{lem:separate symmetry is null} Let $\aa \in \N^H$ and $H \subseteq G$. 
        Define $\aa^o\defeq\aa|_{\dom(\aa)^o}$.  Then, $\FF_\aa^{\sep}=\FF_{\aa^o}$. 
\end{lem}
\begin{proof} It is clear that $\FF_{\aa^o}=\FF_{\aa^o}^{\sep} \subseteq \FF_\aa ^{\sep}$. Similarly to the proof of the previous lemma, one can show that  $\mathcal{F}_\aa^{\sep,n}$ is invariant under the action of a $G$-automorphism $\tau$ if 
        \begin{enumerate}
        \item $\tau$ fixes $\aa^o$, and 
        \item $\tau(\bb)(v)=\bb(v)$ for all $\bb$ and $v$ having some $u \preceq v$ with $\bb(u)>n$. 
        \end{enumerate}
        This shows that $\bigcap_{n \geq 1} \mathcal{F}_\aa^{\sep,n} \subseteq \FF_{\aa^o}$. One can show that $\FF_\aa ^{\sep} = \bigcap_{n \geq 1} \mathcal{F}_\aa^{\sep,n}$ by applying Lemma \ref{lem: symmetric variables with respect to X} to $\bX$, noting that every DAG-exchangeable array is also a \textit{separately} exchangeable array.
\end{proof}

Now we are ready to complete the main task of this subsection, namely, the proofs of Proposition \ref{prop: conditional independence on disjoint graphs, 1} and Lemma \ref{lem: final piece}.

\begin{proof}[Proof of Proposition \ref{prop: conditional independence on disjoint graphs, 1}]
As already noted, by Lemma \ref{lem: symmetric variables with respect to X}, we have
$$
        \FF_C=\sigma(\{\FF_\aa ^{\sep}: \aa \in I_C\})
$$
and
$$
        \FF_{C_i}=\sigma(\{\FF_\aa ^{\sep}: \aa \in I_{C_i}\}).
$$
Now for any $\aa_1 \in I_C$ and $\aa_2 \in I_{C_i}$, we have $\dom(\aa_1 \cap \aa_2) \subseteq C \cap C_i$. Thus, by Proposition \ref{prop: Hoover conditional independence},
$$
        \FF_C \underset{\GG'}{\independent} \FF_{C_1},\ldots, \FF_{C_m} 
$$
where 
$$
        \GG'\defeq\sigma(\{\FF_\aa ^{\sep}: \dom(\aa) \subseteq  C \cap C_i \text{ for some } i \leq m \}).
$$ 
However, Lemma \ref{lem: symmetric variables with respect to X} and Lemma \ref{lem:separate symmetry is null} imply that 
$$
        \GG' = \sigma(\{\FF_\aa : \aa \in I_{C \cap C_i} \text{ for some } i \leq m\}).
$$

\end{proof}

\begin{proof}[Proof of Lemma \ref{lem: final piece}] Fix $\aa \in \N^G$ and $\bb_1,\ldots,\bb_r \in I_G$ such that $\bb_k \neq \aa$. Similar to the proof of Proposition \ref{prop: conditional independence on disjoint graphs, 1}, we can show that
$$S_\aa \underset{\bS_{\srstr(\aa)}}{\independent} S_{\bb_1},\ldots,S_{\bb_r}.$$
Since $\bb_1,\ldots,\bb_r$ are arbitrary, we have
$$S_\aa \underset{\bS_{\srstr(\aa)}}{\independent} (S_\bb : \bb\in I_G, \bb\neq \aa).$$
The joint distributional equality
\eqref{equation: final joint} can now be obtained by applying Lemma \ref{lem: coding lemma} with $\xi_\aa=S_\aa$ and $\eta_\aa=\bS_{\srstr(\aa)}$.
\end{proof}


\appendix

\section{Model-Theoretic Proof of a Simpler Representation Theorem}\label{app:model-theoretic-proof}

Several authors have recently used model-theoretic tools to prove representation theorems for a broad class of exchangeable random structures (e.g.~\cite{ackerman2015,ackermanetal-forum-2016,crane2018relatively,crane2017relative}). 
In the appendix, we prove Theorem \ref{cor 2}, a simplified version of our more general Theorem \ref{cor 2}, using a representation theorem of Crane and Towsner \cite{crane2017relative} which is formulated and proved using model-theoretic tools. We deal with this simplified version only here because we have not yet been able to derive the full version from model theoretic results.

Let $G$ be a finite DAG, $\CC$ a family of closed sets, and ${(\bX_C : C \in \CC)}$ a $\CC$-type random array collection in a Borel space $\XX$. Recall that for each closed set $C$, 
$\indices C$ is the set of $C$-type indices. 

We restate Theorem \ref{cor 2}.

\vspace{2mm}
\noindent{\bf Theorem \ref{cor 2}}\ 
{\it If $\CC$ is the singleton sequence $(G)$ and $(\bX_C : C \in \CC)$ is DAG-exchangeable, then there exists a measurable function $f: [0,1]^{\AA_G} \to \XX$ such that
$$
        \qquad
        \qquad
        \qquad
        \qquad
        \qquad
        \left(X_\alpha : \alpha \in \indices G\right) \ed \Big(f\big(U_{\alpha|_C} : C \in \AA_G\big) : \alpha \in \indices G\Big)
        \qquad
        \qquad
        \qquad
        \qquad
        \eqref{eqn: representation}
$$
        where $\alpha|_C$ is the restriction of $\alpha$ to the vertices in $C$, and the $U_\beta$ are independent $[0,1]$-uniform random variables.
}
\vspace{2mm}
       


\subsection{Review of Crane and Towsner's Representation Theorem}

The following theorem is a minor variant of Crane and Towsner's result in \cite{crane2018relatively}.  In the theorem, we highlight unexplained terminologies with boldface font, to emphasize that we do not expect a reader to understand them at this point.
\begin{thm}[Crane, Towsner]\label{thm: CraneTowsner equivalence relations} Let $\MM=(I,R_1,\ldots,R_n)$ be a countably infinite set $I$ with equivalence relations $R_k$ on it, and $\XX$ a Borel space. Assume that 
\begin{itemize}
        \item $\MM$ is an \textbf{ultrahomogeneous structure}, and
        \item $(R_k : k \leq n)$ is an \textbf{orderly sequence} of equivalence relations.
%
%
%
\end{itemize}
Then, given a family of $\XX$-valued random variables $\bX = (X_\alpha : \alpha \in I)$, if the family is \textbf{relatively exchangeable with respect to $\MM$} (in short, \textbf{$\MM$-exchangeable}), there exists a measurable function $f$ such that
\begin{equation}\label{equation: CraneTowsner representation with equivalence relations}
        \Big(X_\alpha : \alpha \in I\Big) \d= \Big(f\big(U_b : b \in B(\alpha)\big) : \alpha \in I\Big)
\end{equation}
where 
\begin{itemize}
        \item $B(\alpha)$ is the set of all \textbf{anti-chains} in $E(\alpha) \defeq \{[\alpha]_{R_k} ~:~ k\leq n\}\cup\{\{\alpha\}\}$, the collection of all equivalence classes $[\alpha]_{R_k}$ of $\alpha$ with respect to $R_k$'s, partially-ordered by set inclusion, and \item $(U_b : b \in \bigcup_{\alpha \in I} B(\alpha))$ is a collection of independent $[0,1]$-uniform random variables.
\end{itemize}
\end{thm}

\begin{remark} The original theorem \cite{crane2018relatively} has an additional condition that $\MM$ satisfies the  so-called $\omega$-DAP condition up to the $R_k$'s. In this paper, we consider only a special case of the theorem, and in that case, this condition always holds. It is thus omitted in our presentation of the theorem. 
\end{remark}

Most of the boldfaced terms are concepts from model theory. In the rest of this subsection, we explain slightly simplified versions of their definitions. For official definitions and detailed backgrounds of these terminologies, see Crane and Towsner's papers~\cite{crane2017relative,crane2018relatively}.

A \textbf{structure} $\MM$ of type $n$ for some natural number $n$ is a tuple $(I,R_1,\ldots,R_n)$ of a set $I$ and binary relations $\{R_k\}$ on $I$. When another structure $\NN = (J,S_1,\ldots,S_n)$ of the same type satisfies $J \subseteq I$ and $S_k \subseteq R_k$ for all $k$, we say that it is a \textbf{substructure} of $\MM$. A common way of generating a substructure is to restrict $\MM$ with a subset $J_0$ of $I$:
$$
\MM|_{J_0} \defeq ({J_0},\,R_1\cap ({J_0}\times {J_0}),\,\ldots,\,R_n \cap ({J_0} \times {J_0})).
$$
An \textbf{embedding} $\tau$ from a structure $\NN = (J,S_1,\ldots,S_n)$ to a structure $\MM = (I,R_1,\ldots,R_n)$ is a function $\tau : J \to I$ such that $\tau$ is injective and satisfies
$$
\alpha\, [S_k]\, \beta \iff \tau(\alpha)\,[R_k]\,\tau(\beta)\quad
\text{for all $\alpha,\beta \in J$ and all $k \in [n]$.}
$$
Here $[n] = \{m \in \N ~:~ 1 \leq m \leq n\}$.
Note that an embedding from $\NN$ to $\MM$ implies that $\NN$ is essentially the same as $\MM|_{\tau(J)}$,
and provides a sense that $\NN$ is a substructure of $\MM$ modulo renaming of elements of $\NN$.
When the embedding is surjective and $\MM = \NN$, we call $\tau$ an \textbf{automorphism}.

Crane and Towsner used a structure $\MM = (I,R_1,\ldots,R_n)$ with a countably infinite $I$, to specify an index set for a random-variable family and also a symmetry property of that family. The index set is $I$ itself.  They say that a family of random variables $\bX = (X_\alpha : \alpha \in I)$ with this index set is \textbf{relatively exchangeable with respect to $\MM$} or \textbf{$\MM$-exchangeable} if for all finite subsets $J$ of $I$ and embeddings $\tau : \MM|_J \to \MM$,
$$
\tau(\bX) \d= \bX
$$
where $\tau(\bX) \defeq \left(X_{\tau(\alpha)} : \alpha \in I\right)$. Embeddings play the role of finite permutations on $\N$ in the standard notion of exchangeability for random sequences.

Nearly all of the remaining terminology in Theorem \ref{thm: CraneTowsner equivalence relations} describe properties on a structure $\MM = (I,R_1,\ldots,R_n)$. More specifically, they impose requirements on the $R_k$'s, and in doing so, they gauge the $\MM$-exchangeability condition. 

\begin{defi}\label{def: ultrahomogeneity} 
        The structure $\MM$ is \textbf{ultrahomogeneous} if for all finite substructures $$\NN = (J,S_1,\ldots,S_k)$$ of $\MM$ and embeddings $\tau$ from $\NN$ to $\MM$, there exists an automorphism $\upsilon$ on $I$ that extends $\tau$, i.e., $\upsilon|_J = \tau$.
\end{defi}
A representative example of an ultrahomogeneous structure is $(\Q,<)$, the set of rational numbers with the usual less-than relation, while a representative counterexample is $(\Z,<)$, the set of integers with the less-than relation. The latter is not ultrahomogeneous because the function $\tau$ mapping $2$ to $2$ and $3$ to $4$ is an embedding from $(\{2,3\},<)$ to $(\Z,<)$, but cannot be extended to the required global function $\upsilon$ on $\Z$. The lack of any integers strictly between $2$ and $3$ prevents the construction of such an $\upsilon$. The structure $(\Q,<)$ is dense, and does not suffer from this kind of problem. These examples highlight one intuition behind ultrahomogeneity: that $\MM$ does not add any further constraint nor information to that which is present already in an embeddable finite structure.

Our next task is to explain when a sequence $(R_k : k \leq n)$ of equivalence relations of the structure $\MM$ is \textbf{orderly}. Many binary relations on the underlying set $I$ of $\MM$ will appear in our explanation. We call such binary relations simply relations, without mentioning that they are on the set $I$. Also, $R_k$ refers to the $R_k$ of $\MM$. Finally, we remind the reader that $I$ is a countable set and so an equivalence relation on $I$ has only a countable number of equivalence classes.

\begin{defi}\label{def: explicit} 
A relation $R$ is \textbf{basic explicit} in $R_1,\ldots,R_m$ if $R$ has one of the following three forms: 
\begin{itemize}
        \item $R = R_k$ for some $k$; 
        \item $R = \{(\alpha,\alpha) ~:~ \alpha \in I\}$;
        \item $R = I_0 \times I$ or $R = I \times I_0$ for some subset $I_0$ of $I$ that can be defined by a first-order logic formula $\varphi$. The formula $\varphi$ here has one free variable, say $x$, and may use $n$ symbols $r_1,\ldots,r_n$ for binary relations that are interpreted as $R_1,\ldots,R_n$, in addition to the usual quantifiers and logical connectives from first-order logic. This means $I_0 = \{ \alpha ~:~ \text{$\varphi(x)$ holds when $x=\alpha$}\}$.
%
%
%
%
\end{itemize}
A relation is \textbf{explicit} in $R_1,\ldots,R_k$ if it is a Boolean combination of basic explicit relations in $R_1,\ldots,R_k$.
\end{defi}

\begin{defi}\label{def: freely contains} 
An equivalence relation $S$ \textbf{contains} an equivalence relation $R$ if 
$$
        x\,[R]\,y \implies x\,[S]\,y,
$$
or equivalently every equivalence class of $R$ is contained in one of the equivalence classes of $S$. If, in addition, every equivalence class of $S$ contains the same number (possibly countably infinite) of equivalence classes of $R$, we say that $S$ \textbf{evenly contains} $R$, and write $\#_R(S)$ for that number. The relation $S$ is said to \textbf{freely contain} $R$ if $S$ not only evenly contains $R$ but also satisfies the following condition: for all equivalence classes $D$ of $S$, partitions $\{D_1,\ldots,D_m\}$ of $D$ made out of equivalence classes $D_i$ of $R$, and permutations $\pi$ on $[m]$, there exists an automorphism $\upsilon$ on $\MM$ such that\footnote{Here $m$ may be the first countable ordinal, in which case $\pi$ is a permutation on $\mathbb{N}$.}
\begin{itemize} 
        \item $\upsilon(D_k) = D_{\pi(k)}$ for all $k\in [m]$; and 
        \item $\upsilon(D')=D'$ for all the other equivalence classes $D'$ of $R$.
\end{itemize}
\end{defi}
To gain intuition, consider the special case that the structure $\MM$ is $(\mathbb{N}^2,\,R_1,\,R_2)$
with the following equivalence relations $R_1$ and $R_2$:
\begin{equation}\label{eqn:example-R12}
        (k_1,k_2)\,[R_1]\,(k'_1,k'_2) \iff k_1 = k'_1,\qquad
        (k_1,k_2)\,[R_2]\,(k'_1,k'_2) \iff k_1 = k'_1 \wedge k_2 = k'_2.
\end{equation}
Note that $R_2$ is just the equality relation. The relation $R_1$ freely contains $R_2$. It contains $R_2$ because it is a coarser equivalence relation than $R_2$, the equality relation. This containment is even because each equivalence class of $R_1$ contains a countable number of equivalence classes of $R_2$. Checking the remaining condition of free containment is less immediate, but only slightly. Let $D$ be an equivalence class of $R$
and let $\{D_i ~:~ i \in \mathbb{N}\}$ be a partition of $D$
that consists of equivalence classes of $S$. Then, $D$ has the form 
$D = \{(k_0,k) ~:~ k \in \mathbb{N}\}$ for some fixed $k_0$, and each $D_i$ is a singleton set
of the form $\{(k_0,k_i)\}$ for some $k_i$. Given a permutation $\tau$ on $\mathbb{N}$, we may fulfill
the condition of free containment using the following automorphism $\upsilon$ on $\mathbb{N}^2$:
$$
\upsilon(k,k') = 
\begin{cases}
        (k,\tau(k')) & \mbox{if } k = k_0,
        \\
        (k,k') & \mbox{if } k \not= k_0.
\end{cases}
$$
When $k=k_0$ and so $(k,k')$ is in the equivalence class $D$, this function permutes the second component $k'$ according to $\tau$, thus meeting the first bullet point of the condition. Otherwise, $(k,k')$ is not in $D$, and the function acts as the identity, as required by the second bullet point. 

\begin{defi}\label{def: orthogonal} Let $R_1$, $R_2$ be equivalence relations that are contained in an equivalence relation $R$. Then, $R_1$ and $R_2$ are said to be \textbf{orthogonal} within $R$ if for any equivalence classes $D_1$, $D_2$, $D$ of $R_1$, $R_2$, $R$, respectively, with $D_1, D_2 \subseteq D$, we have $D_1 \cap D_2 \neq \emptyset$.
\end{defi}
\begin{defi}\label{def: orderly} 
The sequence $(R_k : k \leq n)$ of equivalence relations is \textbf{orderly} if for each $1\leq k \leq n$, there exists an equivalence relation $R_k'$ such that
\begin{itemize}
    \item $R_k'$ is explicit in $R_1,\ldots,R_{k-1}$;
    \item $R_k'$ freely contains $R_k$; and
    \item if an equivalence relation $S$ is explicit in $R_1,\ldots,R_{k-1}$ and strictly contained in $R_k'$ but it is different from $R_k$, $S$ is either orthogonal to $R_k$ within $R_k'$ or evenly contained in $R_k$ with $\#_{R_k}(S) = \infty$.
\end{itemize}
\end{defi}
A good example of an orderly sequence is $(R_1,R_2)$ made out of relations $R_i$ in \eqref{eqn:example-R12}. The required relations $R'_1$ and $R'_2$ are the complete relation $\mathbb{N}^2 \times \mathbb{N}^2$ and the relation $R_1$, respectively. We focus on $R'_2$. We have already shown that $R_1$ freely contains $R_2$. It is also explicit in $R_1$, simply because it is $R_1$. To check the third condition, consider an equivalence relation $S$ explicit in $R_1$ and strictly contained in $R'_2$. Although we do not present a detailed calculation, it is possible to show that being explicit implies that $S$ has to be one of the following three relations:
$$
        =,\qquad
        R_1,\qquad
        \N^2 \times \N^2.
$$
But only the equality relation is strictly contained in $R'_2$. Thus, $S$ should be the equality relation. That is, $S = R_1$. Our argument so far shows that no $S$ meets the assumptions in the third condition and so the condition holds vacuously.

%
%
%
%

The remaining concept is \textbf{anti-chain}. In a set $A$ with a partial order $\preceq$, an \textbf{anti-chain} is a subset $A_0$ of $A$ such that no two distinct elements of $A_0$ can be compared by $\preceq$, that is, for all $a,b \in A_0$, if $a \neq b$, then neither $a\preceq b$ nor $b \preceq a$. In Theorem~\ref{thm: CraneTowsner equivalence relations}, $A_0$ is a set of certain subsets of $I$ that are equivalence classes of some equivalence relations, and it is ordered by the subset relation.

\subsection{Proof of Theorem \ref{cor 2}}
\label{sec:DAG-exchangeability}

Let $G = (V,E)$ be the DAG in Theorem \ref{cor 2}. Set $n$ to the cardinality of $V$. The first step is to enumerate the vertices of $G$ such that the order in the enumeration respects the directed edges in $E$. We use this enumeration to build a structure $\M$ that has $\N^G$ as its underlying set and satisfies the conditions of Theorem \ref{thm: CraneTowsner equivalence relations}, especially the orderly condition. 

\begin{lem}\label{lem: relabeling} There exists an enumeration of $V$, $(v_\ell : 1 \leq \ell \leq n)$, so that 
$V_\ell \defeq \{v_1,\ldots,v_\ell\}$ is closed for every $\ell \leq n$.
\end{lem}
\begin{proof} This is a well-known simple result. A process for enumerating $V$ is called topological sort in combinatorics and computer science. For completeness, we explain the construction of the sequence $(v_\ell : 1 \leq \ell \leq n)$ in the lemma. We construct the sequence inductively. Since $V$ is finite and $G$ is acyclic, there exists a minimal vertex $v_1$. Our inductive construction starts with the sequence $(v_1)$. Assume that we have enumerated $\ell$ elements such that $V_\ell \defeq \{v_1,\ldots,v_\ell\}$ is closed. Now consider $V \setminus V_\ell$. Since $V$ is finite and partially ordered, so is $V \setminus V_\ell$ and there exists a maximal element $v' \in V \setminus V_\ell$. We set $v_{\ell+1}$ to be this $v'$. Then, by the maximality of $v'$ in $V \setminus V_\ell$, the set $\{v_1,\ldots,v_{\ell+1}\}$ is closed, as required.
\end{proof}

From now on, we write $\M=(I,R_{v_1},\ldots,R_{v_n})$, where the $v_k$ are enumerated as in Lemma \ref{lem: relabeling} and $R_v$ is defined by
$$
\alpha\,[R_v]\,\beta \iff 
\text{for all $w \preceq v$, } \alpha(w) = \beta(w).
$$

\begin{lem}\label{lem: ultrahomogeneity} $\M$ is ultrahomogeneous.
\end{lem}
\begin{proof} 
We use induction on $n$, the cardinality of the vertex set of $G$. For $n=1$, the claim is equivalent to the existence of an extension of a bijection between finite subsets of $\N$ to a permutation of $\N$. So, it is obviously true. Now assume that the claim holds if $n\leq m-1$. We will prove the claim for the case that $n = m$.

Let $\NN = (J,S_1,\ldots,S_m)$ be a substructure of $\MM$, and $\tau$ an embedding from $\NN$ to $\MM$. 
Because of the way that we constructed the enumeration $(v_\ell : 1\leq \ell \leq m)$, the last vertex $v_m$ is maximal according to the partial order induced by $G$. That is, $v_m$ is a terminal vertex. Let $G'$ be the subgraph of $G$ with the vertex set $W=\{v_1,\ldots,v_{m-1}\}$. Let
\begin{align*}
        I|_W & \defeq \N^{W},
        &
        R_{v_\ell}|_W & \defeq \{(\alpha|_W,\alpha'|_W) ~:~ (\alpha,\alpha') \in R_{v_\ell}\},
        \\
        \M|_{W} & \defeq (I|_W,R_{v_1}|_W,\ldots,R_{v_{m-1}}|_W),
        &
        \NN|_{W} & \defeq (J|_W,S_1|_W,\ldots,S_{m-1}|_W),
\end{align*}
Then, $\NN|_W$ is a finite substructure of $\M|_W$. Furthermore, there exists a function $\tau_0 : J|_W \to I_W$ such that $\tau_0(\beta)=\tau(\beta')|_W$ whenever $\beta = \beta'|_W$. In fact, the function $\tau_0$ is an embedding from $\NN|_W$ to $\M|_W$.
By induction hypothesis, $\tau_0$ can be extended to an automorphism $\upsilon_0$ on $\M|_W$.

We now extend $\upsilon_0$ to an automorphism on $\M$. Fix $\beta \in J|_W$. 
Define 
$$
        J_\beta \defeq \{\beta' \in J ~:~ \beta'|_{W}=\beta\}.
$$
Construct a permutation of $\N$, say $\pi_\beta$, so that $\pi_\beta(\beta'(v_m))=\tau(\beta')(v_m)$ for all $\beta' \in J_\beta$. This is possible because $J_\beta$ is finite. Define $\upsilon : I \to I$ as follows:
\begin{equation*}
    \upsilon(\alpha)(v_k) \defeq 
        \begin{cases} 
                \upsilon_0(\alpha|_W)(v_k)
                & \text{if }\, k \leq m-1,
                \\ 
                \alpha(v_k)
                & \text{if }\, k = m\, \text{ and }\, \alpha \notin J,
                \\ 
                \pi_{\alpha|_W}(\alpha(v_k))
                & \text{if }\, k = m\, \text{ and }\, \alpha \in J.
            \end{cases}
\end{equation*}
Then, $\upsilon$ is the desired extension of $\tau$.
\end{proof}

\begin{lem}\label{lem: orderly} The sequence $R_{v_1},\ldots,R_{v_n}$ is orderly.
\end{lem}
\begin{proof} 
For each $k \leq n$, define 
$$
        R_k' \defeq \bigcap \big\{ R_{v_j} ~:~ j<n,\ \, v_j \preceq v_k \,\mbox{ and }\,v_j \neq v_k \big\}.
$$
Clearly, $R_k'$ is explicit in $R_{v_1},\ldots,R_{v_{k-1}}$, and $R_k'$ freely contains $R_{v_k}$. Now consider $S$ such that 
\begin{enumerate}
        \item $S$ is an equivalence relation explicit in $R_{v_1},\ldots,R_{v_{k-1}}$;
        \item $S$ is strictly contained in $R_k'$; and
        \item it is not the case that $S$ is evenly contained in $R_k$ with $\#_{R_k}(S) = \infty$.
\end{enumerate}
A more careful analysis of the equivalence relations explicit in $R_{v_1},\ldots,R_{v_{k-1}}$ for this particular model reveals that they are exactly the equivalence relations that are of the form 
$\bigcap_{i\in I} R_{v_i}$ for ${I\subseteq \{1\dots k-1\}}$. Firstly, the third clause in the notion of basic explicit is redundant on this occasion, for $I_0$ there must be either empty or $I$: these are the only two definable sets. 
Secondly, in this circumstance, if a Boolean combination of relations in $R_{v_1},\ldots,R_{v_{k-1}}$ is an equivalence relation then it must actually be an intersection of such relations; we showed this by considering the disjunctive normal forms that a transitive relation may have in this particular model.

From this we can conclude that $S$ is an intersection of $R_k'$ with some $R_{v_j}$'s where $v_j$ is not an ancestor of $v_k$. Since $\{v_1,\ldots,v_{k-1}\}$ is closed, $v_k$ is not an ancestor of $v_j$ either. Thus, $v_j$ and $v_k$ are incomparable. 
We use this to show that $S$ is orthogonal to $R_{v_k}$ in $R_k'$.
To this end, consider equivalence classes $D, D_1, D_2$ of $R_k'$, $S$, $R_{v_k}$, respectively, with $D_1, D_2 \subseteq D$,
Pick $\alpha_1\in D_1$, $\alpha_2\in D_2$, so that $\alpha_1\mathrel{[R_k']}\alpha_2$,
and let $\beta\in D$ be given by 
$$\beta(v)=\begin{cases}\alpha_1(v)=\alpha_2(v)&\text{if $v\prec v_k$, $v\neq v_k$ };\\
\alpha_1(v)&\text{if $v=v_j$ where $R_{v_j}\subseteq S$ and $v\not\preceq v_k$}\\
\alpha_2(v)&\text{if $v=v_k$}\\
\text{anything}&\text{otherwise}\end{cases}$$
so that $\alpha_1\mathrel {[S]} \beta$ and $\beta\mathrel {[R_{v_k}]} \alpha_2$,
i.e.~$\beta\in D_1\cap D_2$. 
\end{proof}


\begin{proof}[Proof of Theorem \ref{cor 2}] The previous lemmas imply that the conditions of Theorem~\ref{thm: CraneTowsner equivalence relations} hold. Thus, we can apply the theorem, and get the following representation of $\bX$: 
\begin{equation}
        \label{eqn:CT-representation:1}
        \bX = \Big(X_\alpha : \alpha \in \N^G\Big) \d= \Big(f\big(U_b : b \in B(\alpha)\big) : \alpha \in \N^G\Big)
\end{equation}
where $B(\alpha)$ is the set of all anti-chains in $E(\alpha) \defeq \{[\alpha]_{R_k} ~:~ k\leq n\}\cup\{\{\alpha\}\}$, the collection of all equivalence classes $[\alpha]_{R_k}$ of $i$ with respect to the $R_k$'s, partially-ordered by set inclusion, and $(U_b : b \in \bigcup_{\alpha \in \N^G} B(\alpha))$ is a collection of independent $[0,1]$-uniform random variables.

The rest of the proof is about translating the representation in \eqref{eqn:CT-representation:1}
to the claimed representation of Theorem \ref{cor 2}. A crucial part of
this translation is the following function $\varphi$ from $B \defeq \bigcup \{ B(\alpha)  ~:~ \alpha \in \N^G\}$ to $J \defeq \{\alpha|_C ~:~ C \in \AA_G \ \mbox{and}\ \alpha \in \N^G\}$:
$$
        \varphi(b) 
        \defeq 
        \begin{cases}
                \alpha 
                & \text{if $b \in B(\alpha)$ for some $\alpha$ and $b = \{\{\alpha\}\}$}
                \\
                \alpha|_{\{w ~:~ w \preceq v_i \text{ for some $i$}\}} 
                & \text{if $b \in B(\alpha)$ for some/any $\alpha$ and $b = \{[\alpha]_{R_{v_1}},\ldots,[\alpha]_{R_{v_k}}\}$}
        \end{cases}
$$
The function $\varphi$ is well-defined. In the first case of the above definition, there is only one $\alpha$. In the second case, there may be multiple choices of $\alpha$, but they all give rise to the same element in $J$. Furthermore, $\varphi$ satisfies three important properties. Firstly, it is surjective, because for any $C \in \AA_G$ and $\alpha \in \N^G$, we have
$$
        \varphi(\{[\alpha]_{R_v} ~:~ v \text{ is $\preceq$-maximal in $C$}\}) = \alpha|_C.
$$
Secondly, $\varphi$ can be restricted to a surjective function from $B(\alpha)$
to $\{\alpha|_C ~:~ C \in \AA_G\}$ for all $\alpha \in \N^G$.
Finally, it is almost injective in the following sense: when $M$ is the set of $\preceq$-maximal vertices
of~$G$,
$$
        \varphi(b) = \varphi(b') \implies
        \Big(b = b' \, \text{ or }\,
        \big\{b,b'\big\} = \big\{\{\{\alpha\}\},\,\{[\alpha]_{R_v} ~:~ v \in M\}\big\}\,
        \text{ for some $\alpha \in \N^G$}\Big)
$$

Let $g$ be a measurable function from $[0,1]$ to $[0,1] \times [0,1]$ such that for any $[0,1]$-uniform $U$, 
$$
        g(U) \d= (U_1,U_2)
$$
for some independent $[0,1]$-uniform random variables $U_1$ and $U_2$. Pick a collection of independent $[0,1]$-uniform
random variables 
$$
        \bU' \defeq (U'_{\alpha|_C} : \text{$C \in \AA_G$ and $\alpha \in \N^G$}).
$$
Recall that $M$ is the set of $\preceq$-maximal vertices. Let 
$$
        B_0 = \{b \in B ~:~ b = \{\{\alpha\}\}\, \text{ or }\, b = \{[\alpha]_{R_v} : v \in M\} \text{ for some }
\alpha \in \N^G\}.
$$
Then,
\begin{multline*}
        \Big(\big(U_b ~:~ b \in B \setminus B_0\big),\,
        \big(U_b, U_{b'} ~:~ b = \{\{\alpha\}\} \,\text{ and }\, b' = \{[\alpha]_{R_v} : v \in M\} \text{ for some }
        \alpha \in \N^G\big)\Big)
        \\
        {} \d= \Big(\big(U'_{\varphi(b)} ~:~ b \in B \setminus B_0\big),\,
        \big(g(U'_\alpha) ~:~ \alpha \in \N^G\big)\Big)
\end{multline*}
This and the second property of $\varphi$ mentioned above imply the existence of a measurable function $h$ such that
$$
        \Big(\big(U_b : b \in B(\alpha)\big) : \alpha \in \N^G\Big)
        \d=
        \Big(h\big(U'_{\alpha|_C} : C \in \AA_G\big) : \alpha \in \N^G\Big),
$$
which implies 
$$
        \left(X_\alpha : \alpha \in \indices G\right) \ed \Big((f \circ h)\big(U'_{\alpha|_C} : C \in \AA_G\big) : \alpha \in \indices G\Big),
$$
as desired.
\end{proof}

\section{Supplementary results}\label{App: supplementary}
We will sometimes write $(\xi_a)_{a\in I}$ to mean a family of random variables,
and also refer to such a family as an {\it array}. Also, we will use the following notation for conditional distribution properties.
\begin{itemize}
    \item $\xi \underset{\mathcal{F}}{\independent} \eta \qquad\ \ \  (\xi \text{ and }\eta\text{ are conditionally independent given }\mathcal{F}.)$
    \item $\underset{\mathcal{F}}{\independent}(\xi_a)_{a\in I}  \qquad (\text{the family }(\xi_a)_{a\in I}\text{ is conditionally independent given }\FF.)$
\end{itemize}

The first lemma is a standard result from probability theory whose proof we omit.

\begin{lem}\label{lem:Identification} Let $(T_a, V_a)_{a \in I}$ be a multi-indexed family of random variables, and let $\mathcal{F}$ be a $\sigma$-field. Assume the following hold:
\begin{itemize}
    \item $\P[T_a \in \cdot | \FF]= \P[V_a \in \cdot | \FF]$ almost surely
    \item $\underset{\mathcal{F}}{\independent}(T_a)_{a\in I} $
    \item $\underset{\mathcal{F}}{\independent}(V_a)_{a\in I} $
\end{itemize}
Then, $\P[(T_a)_a \in \cdot | \FF] = \P[(V_a)_a \in \cdot | \FF]$ almost surely, and consequently, $(T_a)_a \buildrel d \over = (V_a)_a$.
\end{lem}

The next lemma, which is a simple application of the previous result, is used to synchronize representations using different functions. 
\begin{lem}\label{lem:joining} Let $\xi_0, (\xi_a)_{a \in I}$ be random variables such that $\xi_0 \d= \xi_a$, and let $\bm{\zeta}=(\zeta_a)_{a \in I}$ be a family of independent random variables, which are also independent from $\xi_0, (\xi_a)_{a\in I}$. Let $(\eta_a)_{a\in I}$ be random variables such that for some Borel measurable functions $\phi_a$, the following hold:
\begin{itemize}
    \item $(\xi_0,\eta_a) \buildrel d \over = (\xi_a,\phi_a(\xi_a,\zeta_a))$ for each $a \in I$
    \item $\underset{\xi_0}{\independent}(\eta_a)_{a\in I}$
\end{itemize}
Then, $(\xi_0,\eta_a)_{a\in I} \buildrel d \over = (\xi_0,\phi_a(\xi_0,\zeta_a))_{a\in I}$. In particular, we have $(\xi_0,\eta_a)_{a\in I} \buildrel d \over = (\xi',\phi_a(\xi',\zeta_a))_{a\in I}$ for  any $\xi'\d=\xi_0$ independent from $\bm\zeta$.
\end{lem}
\begin{proof} Since $(\xi_0, \zeta_a) \d= (\xi_a, \zeta_a)$, we can replace $\xi_a$ by $\xi_0$ in the first bullet. \\

Let $T_a=(\xi_0, \eta_a)$, $V_a=(\xi_0, \phi_a(\xi_0, \eta_a))$, $\FF=\sigma(\xi_0)$. Then, $(T_a,V_a)_a$ and $\FF$ satisfy the conditions of Lemma \ref{lem:Identification}. The desired result immediately follows.
\end{proof}

The following coding lemma can be found in \cite[Lemma 7.6]{kallenberg2006probabilistic}:
\begin{lem}\label{lem: coding lemma} Let $(\bm\xi,\bm\eta)=((\xi_a,\eta_a):a \in A)$ be an array with any multi-index set $A$. Assume that
\begin{itemize}
    \item $(\xi_a,\eta_a) \d= (\xi_b,\eta_b)\qquad$ for all $a,b \in A$.
    \item $\xi_a \underset{\eta_a}{\independent} (\xi_b)_{b \neq a}, \bm\eta\qquad$ for all $a\in A$.
\end{itemize}
Then, at the cost of changing the probability space, there exist a Borel function $f$ and an i.i.d. array of uniform random variables $\bm\zeta=(\zeta_a: a \in A)$ such that $\bm\zeta \independent \bm\eta$ and $\xi_a=f(\eta_a,\zeta_a)$ almost surely for all $a\in A$.
\end{lem}

\bibliographystyle{alpha}
\bibliography{bib}

\end{document}